\documentclass[10pt]{article}

\usepackage{graphicx}
\usepackage{dcolumn}
\usepackage{bm}
\usepackage{amsthm}
\usepackage{amssymb}
\usepackage{amsmath}
\usepackage{mathrsfs}
\usepackage{multicol}

\newcommand{\Radon}{\mathcal R}
\newcommand{\Supp}{\mathcal S}
\newcommand{\K}{\mathcal K}

\newcommand{\A}{\mathcal A}
\renewcommand{\H}{\mathcal H}
\newcommand{\T}{\mathcal T}
\newcommand{\R}{\mathbb R}

\renewcommand{\div}{\mbox{\sf div}}
\newcommand{\argmax}{\operatornamewithlimits{argmax}}
\newcommand{\minimize}{\displaystyle \operatornamewithlimits{minimize}}

\usepackage{tikz,pgfplots}
\usepackage{tikzsymbols}
\usetikzlibrary{shapes,snakes}
\usepackage{kpfonts}

\usetikzlibrary{spy,calc}
\usepackage[percent]{overpic}
\usepackage{subcaption}

\newif\ifblackandwhitecycle
\gdef\patternnumber{0}

\pgfplotsset{
    legend entry/.initial=,
    every axis plot post/.code={%
        \pgfkeysgetvalue{/pgfplots/legend entry}\tempValue
        \ifx\tempValue\empty
            \pgfkeysalso{/pgfplots/forget plot}%
        \else
            \expandafter\addlegendentry\expandafter{\tempValue}%
        \fi
    },
}

\pgfkeys{/tikz/.cd,
    zoombox paths/.style={
        draw=orange,
        very thick
    },
    black and white/.is choice,
    black and white/.default=static,
    black and white/static/.style={ 
        draw=white,   
        zoombox paths/.append style={
            draw=white,
            postaction={
                draw=black,
                loosely dashed
            }
        }
    },
    black and white/static/.code={
        \gdef\patternnumber{1}
    },
    black and white/cycle/.code={
        \blackandwhitecycletrue
        \gdef\patternnumber{1}
    },
    black and white pattern/.is choice,
    black and white pattern/0/.style={},
    black and white pattern/1/.style={    
            draw=white,
            postaction={
                draw=black,
                dash pattern=on 2pt off 2pt
            }
    },
    black and white pattern/2/.style={    
            draw=white,
            postaction={
                draw=black,
                dash pattern=on 4pt off 4pt
            }
    },
    black and white pattern/3/.style={    
            draw=white,
            postaction={
                draw=black,
                dash pattern=on 4pt off 4pt on 1pt off 4pt
            }
    },
    black and white pattern/4/.style={    
            draw=white,
            postaction={
                draw=black,
                dash pattern=on 4pt off 2pt on 2 pt off 2pt on 2 pt off 2pt
            }
    },
    zoomboxarray inner gap/.initial=5pt,
    zoomboxarray columns/.initial=2,
    zoomboxarray rows/.initial=2,
    subfigurename/.initial={},
    figurename/.initial={zoombox},
    zoomboxarray/.style={
        execute at begin picture={
            \begin{scope}[
                spy using outlines={%
                    zoombox paths,
                    width=\imagewidth / \pgfkeysvalueof{/tikz/zoomboxarray columns} - (\pgfkeysvalueof{/tikz/zoomboxarray columns} - 1) / \pgfkeysvalueof{/tikz/zoomboxarray columns} * \pgfkeysvalueof{/tikz/zoomboxarray inner gap} -\pgflinewidth,
                    height=\imageheight / \pgfkeysvalueof{/tikz/zoomboxarray rows} - (\pgfkeysvalueof{/tikz/zoomboxarray rows} - 1) / \pgfkeysvalueof{/tikz/zoomboxarray rows} * \pgfkeysvalueof{/tikz/zoomboxarray inner gap}-\pgflinewidth,
                    magnification=3,
                    every spy on node/.style={
                        zoombox paths
                    },
                    every spy in node/.style={
                        zoombox paths
                    }
                }
            ]
        },
        execute at end picture={
            \end{scope}
            \node at (image.north) [anchor=north,inner sep=0pt] {\subcaptionbox{\label{\pgfkeysvalueof{/tikz/figurename}-image}}{\phantomimage}};
            \node at (zoomboxes container.north) [anchor=north,inner sep=0pt] {\subcaptionbox{\label{\pgfkeysvalueof{/tikz/figurename}-zoom}}{\phantomimage}};
     \gdef\patternnumber{0}
        },
        spymargin/.initial=0.5em,
        zoomboxes xshift/.initial=1,
        zoomboxes right/.code=\pgfkeys{/tikz/zoomboxes xshift=1},
        zoomboxes left/.code=\pgfkeys{/tikz/zoomboxes xshift=-1},
        zoomboxes yshift/.initial=0,
        zoomboxes above/.code={
            \pgfkeys{/tikz/zoomboxes yshift=1},
            \pgfkeys{/tikz/zoomboxes xshift=0}
        },
        zoomboxes below/.code={
            \pgfkeys{/tikz/zoomboxes yshift=-1},
            \pgfkeys{/tikz/zoomboxes xshift=0}
        },
        caption margin/.initial=4ex,
    },
    adjust caption spacing/.code={},
    image container/.style={
        inner sep=0pt,
        at=(image.north),
        anchor=north,
        adjust caption spacing
    },
    zoomboxes container/.style={
        inner sep=0pt,
        at=(image.north),
        anchor=north,
        name=zoomboxes container,
        xshift=\pgfkeysvalueof{/tikz/zoomboxes xshift}*(\imagewidth+\pgfkeysvalueof{/tikz/spymargin}),
        yshift=\pgfkeysvalueof{/tikz/zoomboxes yshift}*(\imageheight+\pgfkeysvalueof{/tikz/spymargin}+\pgfkeysvalueof{/tikz/caption margin}),
        adjust caption spacing
    },
    calculate dimensions/.code={
        \pgfpointdiff{\pgfpointanchor{image}{south west} }{\pgfpointanchor{image}{north east} }
        \pgfgetlastxy{\imagewidth}{\imageheight}
        \global\let\imagewidth=\imagewidth
        \global\let\imageheight=\imageheight
        \gdef\columncount{1}
        \gdef\rowcount{1}
        
    },
    image node/.style={
        inner sep=0pt,
        name=image,
        anchor=south west,
        append after command={
            [calculate dimensions]
            node [image container,subfigurename=\pgfkeysvalueof{/tikz/figurename}-image] {\phantomimage}
            node [zoomboxes container,subfigurename=\pgfkeysvalueof{/tikz/figurename}-zoom] {\phantomimage}
        }
    },
    color code/.style={
        zoombox paths/.append style={draw=#1}
    },
    connect zoomboxes/.style={
    spy connection path={\draw[draw=none,zoombox paths] (tikzspyonnode) -- (tikzspyinnode);}
    },
    help grid code/.code={
        \begin{scope}[
                x={(image.south east)},
                y={(image.north west)},
                font=\footnotesize,
                help lines,
                overlay
            ]
            \foreach \x in {0,1,...,9} { 
                \draw(\x/10,0) -- (\x/10,1);
                \node [anchor=north] at (\x/10,0) {0.\x};
            }
            \foreach \y in {0,1,...,9} {
                \draw(0,\y/10) -- (1,\y/10);                        \node [anchor=east] at (0,\y/10) {0.\y};
            }
        \end{scope}    
    },
    help grid/.style={
        append after command={
            [help grid code]
        }
    },
}

\newcommand\phantomimage{%
    \phantom{%
        \rule{\imagewidth}{\imageheight}%
    }%
}
\newcommand\zoombox[2][]{
    \begin{scope}[zoombox paths]
        \pgfmathsetmacro\xpos{
            (\columncount-1)*(\imagewidth / \pgfkeysvalueof{/tikz/zoomboxarray columns} + \pgfkeysvalueof{/tikz/zoomboxarray inner gap} / \pgfkeysvalueof{/tikz/zoomboxarray columns} ) + \pgflinewidth
        }
        \pgfmathsetmacro\ypos{
            (\rowcount-1)*( \imageheight / \pgfkeysvalueof{/tikz/zoomboxarray rows} + \pgfkeysvalueof{/tikz/zoomboxarray inner gap} / \pgfkeysvalueof{/tikz/zoomboxarray rows} ) + 0.5*\pgflinewidth
        }
        \edef\dospy{\noexpand\spy [
            #1,
            zoombox paths/.append style={
                black and white pattern=\patternnumber
            },
            every spy on node/.append style={#1},
            x=\imagewidth,
            y=\imageheight
        ] on (#2) in node [anchor=north west] at ($(zoomboxes container.north west)+(\xpos pt,-\ypos pt)$);}
        \dospy
        \pgfmathtruncatemacro\pgfmathresult{ifthenelse(\columncount==\pgfkeysvalueof{/tikz/zoomboxarray columns},\rowcount+1,\rowcount)}
        \global\let\rowcount=\pgfmathresult
        \pgfmathtruncatemacro\pgfmathresult{ifthenelse(\columncount==\pgfkeysvalueof{/tikz/zoomboxarray columns},1,\columncount+1)}
        \global\let\columncount=\pgfmathresult
        \ifblackandwhitecycle
            \pgfmathtruncatemacro{\newpatternnumber}{\patternnumber+1}
            \global\edef\patternnumber{\newpatternnumber}
        \fi
    \end{scope}
}

\newtheorem{theo}{Proposition}

\title{Phase-Retrieval as a Regularization Problem}
\author{Eduardo X. Miqueles$^1$\footnote{CNPq 442000/2014-6}, Nathaly L. Archilha$^1$, Marcelo R. Dos Anjos$^2$,\\ 
Harry Westfahl Jr.$^1$, Elias S. Helou$^3$\footnote{FAPESP No 2013/07375-0 and CNPq No 311476/2014-7}}

\date{\footnotesize
    $^1$Brazilian Synchrotron Light Laboratory/CNPEM, Campinas, SP/Brazil\\%
    $^2$Federal University of Humait\'a, AM, Brazil \\
    $^3$ICMC/University of S\~ao Paulo/S\~ao Carlos, Brazil     
}

\begin{document}

\maketitle

\begin{abstract}
It was recently shown that the phase retrieval imaging of a sample can be modeled
as a simple convolution process. Sometimes, such a convolution depends on 
physical parameters of the sample which are difficult to estimate {\it a priori}. 
In this case, a blind choice for those parameters usually lead to wrong results, e.g., in posterior image segmentation processing. In this manuscript, we propose a simple connection between phase-retrieval algorithms and optimization strategies, which lead us to ways of numerically determining the physical parameters.
\end{abstract}

\section{Introduction}

Automatic segmentation is a well known tool for extracting information from images, and there are 
many commercial softwares devoted to this branch of image analysis. Even though it is easy to
handle image segmentation using standard tools, there are several examples where such techniques fail.
This is the case of soft-tissue samples exposed to a transmission tomographic
system. Using standard reconstruction algorithms, and dealing with conventional
sinograms for a parallel geometry, it is almost impossible to extract quantitative information,
even though it is easy to visually separate phases in the reconstructed slice. Voxels from these regions contain very similar characteristics and usual segmentation methods fail.
However, it was previously reported by some authors \cite{burvall, sift} that a regularization on 
the measured data noticeably improves the contrast in the final reconstruction.

As an attempt to differentiate the phases, Paganin's \cite{paganin, paganinBook} filter can be applied to the raw data. The approach proposed by Paganin {\it et al} to restore the phase shiftings 
of the object assumes that the sample is placed at a certain distance from
the source. Even then, there is a major difficulty in defining the 
parameters $\{\delta, \beta\}$ for some samples, which compose the refractive 
index $n = 1 - \delta + i \beta$ and play a major role on the numerical scheme. A major assumption on phase coefficients $\{\delta,\beta\}$ is that $\delta$ is proportional to $\beta$
for all projection images. Our contribution on this manuscript is to provide 
simple forms to obtain the ratio $\delta / \beta$ using a simple one-dimensional
optimization routine. This is almost immediate after recognizing the approach proposed by Paganin as a Sobolev regularization \cite{sobolev}. In fact, in our proposed technique, the ratio $\delta / \beta$ comes from approach similar to the one used in the L-curve \cite{hansen}.

\medskip

\noindent \paragraph*{\bf Scientific case}\ Besides representing a major source of animal protein in the Amazon forest, fishes play a critical role in this ecosystem, mainly due to the exceptionally large bowl of the river, which allow them to interact along regional space, across various trophic levels \cite{lowe}. Freshwater fishes are represented by around 8,500 species (over 40\% of all known species of fish), mostly in the vast river systems and tropical lakes \cite{cohen}. The interaction between the fish fauna with aquatic ecosystems and biota takes place through food interrelationships and effects on the chemical composition of the water and sediment.

Freshwater fish of the Amazon synchronize their reproductive period with the period of high waters, however, studies presented in recent years in the Intergovernmental Panel on Climate Change have shown global climate change, which has greatly affected the reproduction of this group leading to decline of fish stocks in the region as well as the associated biotopes and interdependence of these. Also, the group of fish works as an excellent environmental indicator showing responses of changing state of normalcy that are likely to be measured. Thus, it is desirable to develop a protocol for the acquisition of {\it Prochilodus nigricans} (popular known as {\it Curimat\~a} in native language) eggs screenshots which emphasize possible responses in the reproductive dynamics of Amazonian fish populations and possible changes due to environmental changes on a global scale. These images could be used in studies eventually leading to proposals that could even steer public policy in the period closed in order to contribute to the conservation of fish stocks and associated biodiversity.
\begin{figure}[h]
\centering
\includegraphics[scale=0.15]{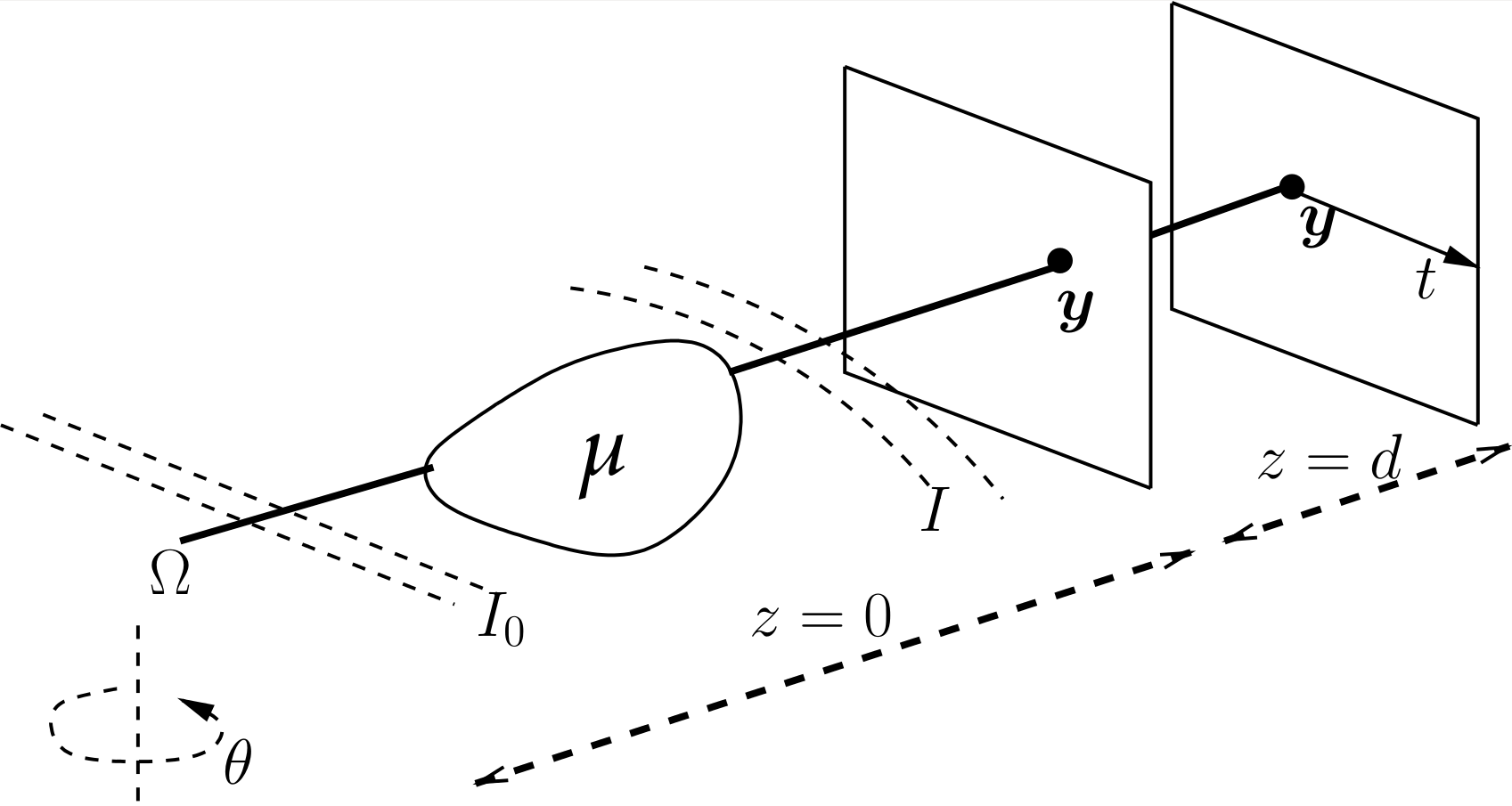}  
\caption{Paraxial propagation describing the transport of intensity  of a scalar electromagnetic wavefield through a sample.}
\label{fig:wave}
\end{figure}
Since the sample is mainly composed by soft tissue determined by approximately one single material, all measurements were done using a phase contrast setup. In the present paper, the focus is on the mathematics of the phase contrast tomographic reconstruction, in order to enable the development of such a protocol in the future. To preliminarily test our techniques, we expose a stained \cite{cardoso} sample of fish eggs on the imaging beamline of the Brazilian Synchrotron Laboratory aiming at exploring the morphology of the Curimat\~a sample using conventional micro-tomography.

\section{Mathematical Description} 

Mathematically, an intensity $I$ is measured at a given pixel $\bm y$ on the plane detector, placed after the sample $\mu$; see Fig.\ref{fig:wave}. Along the raypath $\Omega$, the paraxial wavefront of the scalar electromagnetic field is described by the transport intensity equation (\textsc{tie})
\begin{equation} \label{eq:TIE}
\div ( I (\bm y,z) \nabla \phi (\bm y,z)) = \frac{\partial I}{\partial z}(\bm y)
\end{equation}
where $\phi$ indicates the phase and $\bm y$ is a point on the detector plane. A large class of methods for the phase-retrieval problem is summarized in the work of Burvall \cite{burvall}. It is claimed that a general approximate solution of the \textsc{tie} (\ref{eq:TIE}) at the near contact plane $z=0$ is proportional to the thickness of the sample along $\Omega$ -- say $p = p(\bm y)$ -- and must satisfy the following functional relation
\begin{equation} \label{eq:prop}
(-\ell \nabla^2 + 1) e^{-p(\bm y)} = f(\bm y), \ \ f(\bm y) := \frac{I(\bm y)}{I_0(\bm y)}
\end{equation}
with $\nabla^2$ standing for the Laplacian operator and $\ell$ a constant
depending on the phase coefficients $\{\delta, \beta\}$ and the distance $d$. Here, $I_0$ is the incident wavefront at the sample. After taking the Fourier transform on the above equation, the approximate solution $p$ becomes
\begin{equation} \label{eq:paganin}
p(\bm y) = \mathcal M \left[ (\K_\ell \star f) (\bm y) \right] := \mathscr P_\ell (f).
\end{equation}
Operator $\mathscr P_\ell$ indicates the operator described by Paganin, for a fixed $\ell$. In this work, we are particularly interested in the operator $\mathcal M$ and kernel $\K_\ell$ described by
\begin{equation}
\mathcal M[t] = - c \ \ln t, \ \ \widehat{\K}_\ell(\bm q) = \frac{1}{1+ 4\pi^2 \ell \|\bm q\|^2}   
\end{equation}
where $\hat{\cdot}$ stands for Fourier transformation,
\begin{equation}
    \hat{f}(\bm q) = \int_{\R^2} f(\bm y) e^{-i \bm q \cdot \bm y} \mathrm d \bm y
\end{equation}
For example, in the case of an object with a single material, the approximate solution was given by Paganin \cite{paganin} with $c = 1$ and $\ell = d \delta / \mu$. Taking $\ell=0$, the solution of (\ref{eq:prop}) is obtained using the inverse of the Radon Transform $\Radon$. For this case, there are several reconstruction  algorithms able to invert the operator $\Radon$.

Phase-retrieval discussion starts with the assumption 
that the integration operation $\Radon$ along the sample $\mu$ -- through  the raypath $\Omega$ and hiting the detector at point $\bm y$ -- can be approximated by a mean value $c$. That is
\begin{equation}
e^{-\Radon[\mu](\bm y)} \approx e^{- c \Supp[\mu](\bm y)}
\end{equation}
where $\Supp[\mu]$ denotes the support (thickness) of the sample $\mu$ along 
the ray parameterized by $\bm y$. Hence, the general strategy in the literature, as pointed out in the work of Burvall~\cite{burvall}, is that $p(\bm y) = c\Supp[\mu](\bm y)$
is the approximate solution of the \textsc{tie} (\ref{eq:TIE}), which satisfies 
the following equation
\begin{equation} \label{eq:1st}
f(\bm y) = e^{-\Radon[\mu](\bm y)} \approx e^{-p(\bm y)} = (\K_\ell \star f)(\bm y)
\end{equation}
First equality on the above equation comes from the standard Beer-Lambert law. Using properties of the Fourier transform, the kernel $\K_\ell$ can be written on the detector domain as
\begin{equation}
\K_\ell(\bm y) = \frac{\pi}{\sqrt{\ell}}  e^{-2\pi \|\bm y\|/\sqrt{\ell}}
\end{equation}
Hence, obtaining the map $p$ is a straightforward numerical process by means of the Fast Fourier Transform. Here, the constant $\ell$ comes from a strong physical assumption on the sample, which either relates to
a single or a double material.  From (\ref{eq:1st}), the approximation $f \approx e^{-p}$ indicates that in some sense, $f - e^{-p}$ should be minimized in $p$. It is obvious, due to the Beer-Lambert law for absorption, that the solution of $f = e^{-p}$ return $p$ as the Radon transform of the attenuation coefficient, which in turns implies that we are not restoring the phase $\phi$ at the plane $z=0$. 

\section{Tikhonov Regularization {\it versus} Phase Retrieval}
\label{sec:tiko}

In the Schwartz space $S(\R^2)$ equipped with the $L_2$ norm $\|u\|^2 = \langle u, u\rangle$, let us denote $u = e^{-p}$. We look for a map $u$ which is 
the best smooth approximation of the normalized map $f$. Smoothness in the solution $u$ is evaluated by the use of a ``roughness measurement'' operator $\A$. This leads to the problem of finding a global minimizer of the following functional
\begin{equation}\label{eq:min}
    \begin{array}{lll}
       \minimize_{u}\quad \|  u - f \|^2 + \lambda \|\A u\|^2 \\
       \end{array}
\end{equation}
with a regularization factor $\lambda$. Indeed, (\ref{eq:min}) is 
a standard Tikhonov regularization approach for the smoothed least
squares solution $u$. A typical definition for the functional $\A$, enforcing smoothness on the solution $u$, is the following
\begin{equation} \label{eq:f_quad}
\| \A u \|^2 = \int_{D} \|\nabla u(\bm y)\|^2 \mathrm d \bm y
\end{equation}
The first order optimality condition for the unconstrained problem (\ref{eq:min}) is determined by the Euler-Lagrange equations, which in the case of (\ref{eq:f_quad}), coincide with the following normal equations
\begin{equation} \label{eq:normal}
u - f + \lambda \A^* \A u = 0 \iff (\lambda \A^* \A + 1) u = f
\end{equation}
where $\A^*$ is the hermitian-adjoint operator of $\A = \nabla$. Due to 
the fact that $\nabla^* = - \div$ and $\div \ \nabla = \nabla^2$, 
the optimality condition (\ref{eq:normal}) becomes
\begin{equation} \label{eq:oc}
( - \lambda \nabla^2 + 1 ) u = f
\end{equation}
which is essentially the same equation as (\ref{eq:prop}), 
obtained by approximate phase retrieval methods, 
with $\lambda = \ell$. In this case, it means that the  
optimal regularization parameter $\lambda$ is the one 
that obeys some physical assumption. 


\begin{theo} \label{prop0}
The in-line phase-Retrieval approximated solution $p=p(\bm y)$ (a projection image) proposed by Paganin et al -- described in (\ref{eq:paganin}) and with initial data $f = f(\bm y)$ -- is a local minimizer of the regularized problem
\begin{equation} \label{eq:opt}
(\textsc{p})\colon \ \ \ \minimize_p \ \H(p) :=  \|e^{-p} - f\|^2 + \ell \int_{\R^2} \|\nabla e^{-p(\bm y)} \|^2 \mathrm d \bm y
\end{equation}
for a fixed parameter $\ell$. The analytical solution of problem (\textsc{p})  is determined by (\ref{eq:paganin}).
\end{theo}

\begin{proof}
Let us denote the function to be minimized by $H(p) = E(p) + \ell R(p)$, with
\begin{equation} \label{eq:fun_E_R}
E(p) = \| e^{-p} - f \|^2, \ \ \ R(p) = \int_{\R^2} \left\|\nabla \left( e^{-p(\bm y)} \right) \right\|^2 \mathrm d \bm y  \end{equation}
The Fr\'echet derivative of nonlinear functionals $R$ and $E$ are presented in 
the Appendix, using $f = e^{-g}$. The action of the linear functional 
$H'(p)$ on each $h \in L^2$ is determined by
\begin{equation}
\begin{array}{lll}
H'(p)h = 2 \displaystyle \int e^{-p}(e^{-p}-e^{-g}) h \ \mathrm d\bm y - 2 \displaystyle \ell \int e^{-p} \nabla^2 \left( e^{-p} \right) h \ \mathrm d \bm y \\
\ \ \ = \displaystyle 2 \int e^{-p} [ e^{-p} - e^{-g} - \ell \nabla^2 (e^{-p}) ] h \mathrm d \bm y
\end{array}
\end{equation}
Hence, $H'(\bar{p}) = 0$ if and only if $\bar{p}$ satisfies the equation
\begin{equation} \label{eq:stationary}
e^{-\bar{p}}-e^{-g} - \ell \nabla^2 \left( e^{-\bar{\bar{p}}} \right) = 0,
\end{equation}
which using $u = e^{-\bar{p}}$ becomes the same as (\ref{eq:oc}). That is, the solution of
the Paganin equation is a stationary point of the functional $H(p)=E(p) + \ell R(p)$.

To prove that $p$, satisfying (\ref{eq:stationary}), is a local minimizer, we have to 
show that the second Fr\'echet derivative $H''(p)$, as a bilinear form, satisfies
\begin{equation} \label{eq:CSO}
H''(\bar{p}) (v,v) \geq \alpha \|v\|^2, \ \ \forall v \in L^2 
\end{equation}
(see \cite{luenberger}). The second Fr\'echet  derivative of each functional $E$ and $R$ are presented in the Appendix. Form $E''(p)$ evaluated at a pair $(v,v) \in L^2 \times L^2$ obeys
\begin{equation}
\begin{array}{lll}
E''(\bar{p})(v,v) = \displaystyle -2 \int e^{-\bar{p}} [ e^{-g} - e^{-\bar{p}} - e^{-\bar{p}}] v^2 \mathrm d \bm y  \\
\ \ \ =  \displaystyle -2 \int e^{-\bar{p}} [ e^{-g} - e^{-\bar{p}}] v^2 \mathrm d \bm y  + 2 \int e^{-2\bar{p}} v^2 \mathrm d \bm y  \\
\ \ \ = \displaystyle 2\ell\int e^{-\bar{p}} \nabla^2(e^{-\bar{p}}) v^2 \mathrm d \bm y  + 2 \int e^{-2\bar{p}} v^2 \mathrm d \bm y
\end{array}
\end{equation}
The second Fr\'echet derivative of $R$, at the stationary point $\bar{p}$ follows
\begin{equation}
R''(\bar{p})(v,v) = 2 \int v e^{-\bar{p}} \nabla^2 (e^{-\bar{p}} v) + v^2 e^{-\bar{p}} \nabla^2 (e^{-\bar{p}}) \mathrm d \bm y.
\end{equation}
The bilinear form $H''(p) = E''(p) + \ell R''(p)$ at $p = \bar{p}$, acting on $(v,v)$ -- 
after several cumbersome but elementary simplifications -- will be given by
\begin{equation} \label{eq:almost}
\begin{array}{lll}
H'(\bar{p})(v,v) =  \displaystyle 4 \ell \int e^{-\bar{p}} v^2 \nabla^2 (e^{-\bar{p}} ) \mathrm d \bm y +  
2 \int v^2 e^{-2\bar{p}} \mathrm d \bm y   \\
\ \  + 2 \ell \displaystyle \int v e^{-\bar{p}} \nabla^2 (e^{-\bar{p}} v) \mathrm d \bm y
\end{array}
\end{equation}
Since $\bar{p}$ is bounded (due to the compactness of the object) function $e^{-\bar{p}}$ and $\nabla^2 (e^{-\bar{p}})$ will also be minorized by some constant, due to the fact that $e^{-\bar{p}} \geq M_1$. Let us assume that $\nabla^2 (e^{-\bar{p}}) \geq M_2$. Using this fact, we can minorize (\ref{eq:almost}) 
\begin{equation} \label{eq:almost2}
\begin{array}{lll}
H'(\bar{p})(v,v) \geq  \displaystyle 4 \ell M_1 M_2 \|v\|^2 +  2 M_1^2 \|v\|^2 \\
\ \ \ \ \ \ \  + 2 \ell \displaystyle \int v e^{-\bar{p}} \nabla^2 (e^{-\bar{p}} v) \mathrm d \bm y
\end{array}
\end{equation}
To minorize the last integral, we use the fact for all $v$ such that $\|v\| \leq \varepsilon$, 
$v \nabla^2 v \sim 2 v^2$ and $v \nabla v \sim v^2 \bm 1$ with $\bm 1 = (1,1)\in\R^2$.
Hence,
\[
\begin{array}{lll}
v \nabla^2 (e^{-\bar{p}}v ) = v^2 \nabla^2(e^{-\bar{p}}) + 2 \nabla(e^{-\bar{p}})\cdot (v \nabla v) + e^{-\bar{p}} (v\nabla^2 v) \\
\ \ \ \sim \left[ \nabla^2(e^{-\bar{p}}) + 2 \nabla(e^{-\bar{p}})\cdot \bm 1 + 2 e^{-\bar{p}} \right] v^2
\end{array}
\]
Expression between brackets in the above equation has a lower bound $M_3$ since 
$p$ has bounded variation. Now, (\ref{eq:almost2}) becomes
\begin{equation} \label{eq:temerViado}
H'(\bar{p})(v,v) \geq  (4 \ell M_1 M_2 + 2 M_1^2 + 2 \ell M_3 M_1) \|v\|^2
\end{equation}
Since $H$ is twice continuously Fr\'echet-diferentiable, the Taylor expansion
around $\bar{p}$ gives 
\[
H(p) = H(\bar{p}) + H'(\bar{p})(p - \bar{p}) + \frac{1}{2} H''(\bar{p}) (p-\bar{p}, p-\bar{p}) + o(\|p-\bar{p}\|^2)
\]
for all $p$ such that $\|p-\bar{p}\| < \varepsilon$. Using (\ref{eq:temerViado}) 
and $H'(\bar{p}) = 0$ we finally obtain $\bar{p}$ as the local minimizer of $H$.
\end{proof}

Notice that, as $\ell$ increases, the retrieved \emph{phase} $p$ becomes a blurred version of the normalized image $f$; on the other hand, taking $\ell=0$ gives us the linear attenuation coefficient. Assuming that approximate phase-retrieval methods give a solution of the optimization problem (\ref{eq:min}), there are many numerical strategies to find an ideal parameter \cite{sift,elias} $\ell$, e.g., the $L$-curve~\cite{gulrajani}. Even though $\ell$ is strictly related to physical properties of the sample under investigation, the regularized solution (\ref{eq:paganin}) $p_\ell$ is the one maximizing the contrast of the solution. In fact, from the
optimality condition (\ref{eq:oc}), the following decreasing function
\begin{equation}
\| \nabla^2 (e^{-p_{\ell}}) \|^2 = \frac{\| e^{-p_{\ell}} - f\|^2}{ \ell^2}
\end{equation}
always attains a maximum curvature at an optimal point $\ell^*>0$. This is true due to the property of the $L$-curve\cite{hansen}, which states that a compromise must be taken between the minimization of $\|u - f\|$ and smoothness by operator $\A$ through the plot
of the curve $(\|u_\ell - f\|, \|\A u_\ell\|)$. Hence, we consider $\ell = \ell^*$ as the optimal parameter satisfying
\begin{equation} \label{eq:curvs}
\ell = \displaystyle \argmax_{\lambda\geq 0}\ \kappa_{\sf f}(\lambda); \ \  
\end{equation}
and $\kappa_{\sf f}(\ell)$ being the curvature of the function
$\ln \xi_{\sf f}(\ell)$ with $\xi_{\sf f} = \|\nabla^2 p_\ell \|_*^2$, defined by
\begin{equation} \label{eq:kappa}
\kappa_{\sf f}(\ell) = \frac{1}{\xi_{\sf f}(\ell)^2}\frac{|\xi_{\sf f}''(\ell) \xi_{\sf f}(\ell)-(\xi_{\sf f}(\ell)')^2|}{\sqrt{1+\left[\frac{\xi_{\sf f}'(\ell)}{\xi_{\sf f}(\ell)}\right]^2}^3} 
\end{equation}
with prime denoting derivative with respect to $\ell$ (subscript $\sf f$ stands for \emph{frame}
since this is a process done by projections). Computing
first and second derivatives from function $\xi_{\sf f}$ is a straightforward process due to properties of the Fourier transform. In fact, since $e^{- p_\ell} = G_\ell$ with
$G_\ell = \K_\ell \star f$ from (\ref{eq:paganin}), it is easy to obtain
\begin{equation} \label{eq:K1}
\widehat{\frac{\partial^n e^{-p_\ell}}{\partial \ell^n}}(\bm q) =
\frac{\partial^n}{\partial \ell^n} \widehat{e^{-p_\ell}}(\bm q) = \frac{n! (-1)^n (4\pi^2\|\bm q\|)^{2n}}{(1 + 4\pi^2\ell \|\bm q\|^2)^{n+1}} \hat{f}(\bm q)
\end{equation}
which implies that 
\begin{equation}
\frac{\partial^n [e^{-p_\ell(\bm y)}]}{\partial \ell^n} = (\K^{(n)}_\ell \star f)(\bm y), \ \ n=0,1,2
\end{equation}
with kernels $\{\K^{(1)}_\ell, \K^{(2)}_\ell\}$ defined in the frequency domain by equation (\ref{eq:K1}). Computing the optimal $\ell$ with an optimization algorithm for one-dimensional functions (such as Newton's method or \textsc{bfgs}) is fast since the evaluation of the curvature 
function $\kappa$ depends only on Fourier transforms.

One of the main disadvantages of the Paganin filtering approach, is that each frame has to be processed individually in order to obtain a single filtered block of the raw data. From the new dataset, we extract two or three dimensional reconstructions using some inversion scheme. This is an extremely intensive computational process that can be done with a graphics processing unity to speed up calculations \cite{maia}. To avoid \emph{nested loops} in a programming strategy of in-line phase retrieval method, we propose a filtering strategy directly on the sinogram. In fact, we can use the Lipschitz property of the exponential function
\begin{equation} \label{eq:lips}
|e^{-a} - e^{-b}| \leq |a-b|, 
\end{equation}
In our case, the unknown $p$ approximates the path integral of the sample -- say $g = -\ln f$ -- (i.e., the Radon transform of $\mu$) on pixel $\bm y$ of the area detector. Therefore, minimization of $\|p - g\|_*$ implies minimization of $\|e^{-p} - e^{-g}\|$ due to (\ref{eq:lips}). Therefore, minimization of $\|p - g\|^2$ implies minimization of $\|e^{-p} - e^{-g}\|^2$. This is true because the
Fr\'echet derivative of the function $E(p) = \|e^{-p} - e^{-g}\|^2$ is 
the linear functional $E'(p) \colon L^2 \to \R$ determined by
\begin{equation} \label{eq:frechet}
\frac{1}{2} E'(p) h = \int_{\bm y \in D} e^{-p(\bm y)} \left( e^{-g(\bm y)} - e^{-p(\bm y)} \right) h(\bm y) \mathrm d \bm y
\end{equation}
Hence, $p=g$ is a critical point of the functional $E$ since $E'(g) = 0$ (first order optimality condition). Since $E$ is convex, $p=g$ must be a global minimizer. A proof of (\ref{eq:frechet})
is given in the appendix.

Another strategy for phase-retrieval, the so-called \emph{Fourier method} obtained with Born 
or the Rytov approximation described in \cite{rytov}, claims that 
a solution $p$ is obtained using 
\begin{equation} \label{eq:brytov}
p = \T_\ell * \left[ - \ln f \right]
\end{equation}
with $\T_\ell$ a suitable kernel described in \cite{burvall}. The main  difference of
(\ref{eq:brytov}) and (\ref{eq:paganin}) is the filter action, which is performed 
after the log-normalization. In this sense, we look for a smooth function $p$ which is close to $g$, i.e.,
a solution of the following regularized optimization problem
\begin{equation} \label{eq:reg4}
\minimize_{p} \ \mathcal V (p) := \| p - g\|^2 + \ell \| \mathcal Y p\|^2
\end{equation}
with $\mathcal Y$ a smoothing operator. The Euler-Lagrange
equations for the above functional $\mathcal V$ 
can be simplified considering the skew-adjoint operator $\mathcal Y = \frac{\partial}{\partial_t}$ (here, $t$ coincide with the $y_1$ axis, the horizontal axis for a given slice)
\begin{equation} \label{eq:normeq}
(\ell \mathcal Y^* \mathcal Y + 1) p = g \iff \left(-\ell \frac{\partial^2}{\partial_t^2} + 1\right) p = g
\end{equation}
After taking Fourier transformations, we arrive at the optimal regularized solution
\begin{equation} \label{eq:pr2}
\hat{p}(\sigma, \theta) = \left(\frac{1}{1 + 4\pi^2 \ell \sigma^2}\right) \hat{g}(\sigma, \theta) = \widehat{\T_\ell}(\sigma) \hat{g}(\sigma,\theta)
\end{equation}
which is a filtering operation over the {\it sinogram} $g$, i.e.,
$p(t,\theta) = (\T_\ell \star g)(t,\theta)$. The retrieved sinogram $p$  
obtained with (\ref{eq:pr2}) is similar to the Rytov-Born \cite{rytov} approximated solutions for in-line phase retrieval tomography in the Fresnel regime. Both assume that phase coefficients $\{\delta, \beta\}$ are proportional, and also that $\beta \approx 0$ which is true for the fish egg sample. 

\begin{theo} \label{prop1}
The in-line phase-retrieval approximated solution $p=p(t,\theta)$ (a sinogram image) proposed by the Fourier method - described in (\ref{eq:brytov}), (\ref{eq:pr2}) and with initial data $g = g(t,\theta)$ - is a local minimizer of the regularized problem
\begin{equation}
(\textsc{r})\colon \ \ \ \minimize_p \|p - g\|^2 + \ell \int_{\R} \int_0^\pi \left(\frac{\partial p(t,\theta)}{\partial t}\right)^2 \mathrm d t \mathrm d \theta
\end{equation}
\end{theo}
\begin{proof}
This an immediate consequence of (\ref{eq:reg4}) using $\mathcal Y = \partial / \partial t$, which lead us to
the normal equations (\ref{eq:normeq}) and thus to a stationary point $p$ given by (\ref{eq:pr2}). It is a local minimizer since the second Fr\'echet of the objective functional is a positive semi-definite bilinear form, with $\alpha = 1 - \ell$ (see equation (\ref{eq:CSO})).   
\end{proof}

In this particular case, the optimal parameter $\ell$ can be found as the one 
that maximizes the curvature $\kappa_{\sf s}$ of the function $\ln \xi_{\sf s}(\ell)$ with 
\begin{equation} \label{eq:xis}
\xi_{\sf s}(\ell) = \| \partial^2_t p_\ell \|^2, \ \ \ p_\ell(t) = (\mathcal T_\ell \star g)(t,\theta)
\end{equation}
where $\partial_t^2$ stands for the second derivative on the \emph{ray} axis $t$ (subscript $\sf s$ stands for
\emph{slice} since this is a process done by sinograms). Curvature function $\kappa_{\sf s}$ is the 
same defined in (\ref{eq:kappa}), replacing $\xi_{\sf f}$ by $\xi_{\sf s}$. As stated previously,
a one-dimensional optimization algorithm can easily take advantage of the fast evaluation of the objection
function $\kappa_{\sf s}$ through fast fourier transforms. For completeness, it is easy to note  that
\begin{equation}
\frac{\partial^n \left[ p_\ell(t,\theta) \right] }{\partial t^n} = (\mathcal T^{(n)}_\ell \star g)(t,\theta), \ \ n=0,1,2
\end{equation}
with $\mathcal T^{(n)}_\ell$ being represented in the frequency domain through
\begin{equation}
\widehat{\mathcal T^{(n)}_\ell}(\sigma,\theta) = \frac{n! (-1)^n (4\pi^2\sigma^2)^{2n}}{(1 + 4\pi^2 \ell \sigma ^2)^{n+1}} 
\end{equation}

The main difference for the in-line phase 
retrieval methods described by Proposition \ref{prop1} relies on the evaluation of the difference 
\begin{eqnarray} \label{eq:delta}
\Delta_\ell(\bm y) &=& - ( \T_\ell \star \ln f ) (\bm y) + \ln \left( \K_\ell \star f \right)(\bm y) \\
&=& ( \T_\ell \star [-\ln f]) (\bm y) + \ln \left( \K_\ell \star f \right)(\bm y) 
\end{eqnarray}
We propose the following result for an upper bound estimate to function $\Delta$.

\begin{theo} \label{prop:diff}
The difference (\ref{eq:delta}) obtained by in-line phase retrieval methods -
described in Proposition \ref{prop1} - is bounded by $\frac{\|\ln f\|}{1+4\pi^2\ell}$ for all $\ell>0$.
For $\ell=0$ the sinograms are equal.
\end{theo}
\begin{proof}
Case $\ell=0$ is trivial since
\[
\lim_{\ell\to 0} \K_\ell \star f = f = \lim_{\ell \to 0} \T_\ell \star f \ \ \Rightarrow
\ \ \lim_{\ell \to 0} |\Delta (\bm y)| = 0
\]
Since $\K_\ell$ and $\T_\ell$ are positive operators, it is true from the triangle 
inequality that
\begin{eqnarray}  
|\Delta_\ell(\bm y)| &\leq&  |\T_\ell \star \ln f)( \bm y)| + | \ln \left( \K_\ell \star f \right)(\bm y)|  \\
&\leq & (\T_\ell \star |\ln f |)(\bm y) - \ln \left( \K_\ell \star f \right)(\bm y) \label{eq:jensen0}
\end{eqnarray}
Last inequality holds since $f < 1$ and $\K_\ell \star f \leq 1$. Due to Jensen's 
inequality \cite{tristan} we can majorize the second convolution 
integral of (\ref{eq:jensen0}) to obtain
\begin{eqnarray} 
|\Delta_\ell(\bm y)| &\leq&  (\T_\ell \star |\ln f|)( \bm y)  - \left( \K_\ell \star \ln f \right)(\bm y) \\
&\leq& (\T_\ell \star |\ln f|)( \bm y)  + \left( \K_\ell \star [-\ln f] \right)(\bm y) \\ \label{eq:jensen}
&\leq& \|\T_\ell \star |\ln f| \| + \| \K_\ell \star |\ln f| \| \\
&\leq& \|\T_\ell \|_2 \| \ln f \| + \| \K_\ell \|_2 \|\ln f \|
\end{eqnarray}
Finally, using Parseval's identity, each $L^2$ norm $\|\hat{\K}_\ell\|$ and $\|\hat{\T}_\ell\|$
is assymptotically bounded in variable $\ell$ by the factor $\frac{1}{1+4\pi^2\ell}$
because the maximum of $\hat{\K}_\ell$ and $\hat{\T}_\ell$ is one, in variables $\bm q$
and $\sigma$, respectively. In this case,
\begin{equation}
\max_{\bm y} |\Delta_\ell(\bm y)| \leq \frac{2\|\ln f\|}{1+4\pi^2 \ell}
\end{equation}
is the final bound for the difference (\ref{eq:delta}). 
\end{proof}

It is important to note that the one-dimensional convolution integral with kernel $\T_\ell$ applies over the $\bm y_1$ axis, 
keeping $\bm y_2$ fixed (the slice axis). On the other hand, the two-dimensional integral of the denominator
- with kernel $\K_\ell$ - applies for the entire domain of $\bm y$. To overcome the difficulty
in comparing kernels $\T_\ell$ and $\K_\ell$, let us introduce the operator $\nabla^2_\epsilon$
\begin{equation}
\nabla^2_\epsilon = \frac{\partial^2}{\partial \bm y_1^2} + \epsilon \frac{\partial^2}{\partial \bm y_2^2}  
\end{equation}
It is obvious that $\nabla^2_\epsilon \to \nabla^2$ as $\epsilon \to 1$ and Fourier transforms results
in $\mathscr{F}[\nabla^2_\epsilon f] = -4\pi^2 [\bm q_1^2 + \epsilon \bm q_2^2 ] \mathscr{F}[f]$.
Returning to starting equation (\ref{eq:prop}) and replacing $\nabla^2$ by $\nabla^2_\epsilon$
we arrive at an equation similar to (\ref{eq:paganin}) with a kernel $\K_{\ell,\epsilon}$ defined
in the frequency domain by
\begin{equation}
\widehat{\K_{\ell,\epsilon}}(\bm \omega) = \frac{1}{1 + 4\pi^2\ell[\bm q_1^2 + \epsilon \bm q_2^2]}
\end{equation}
Now, it is easy to realize that $\widehat{\T}_\ell = \displaystyle \lim_{\epsilon \to 0} \widehat{\K}_{\ell,\epsilon}$.

\begin{figure}[h]
    \centering
    \includegraphics[scale=0.5]{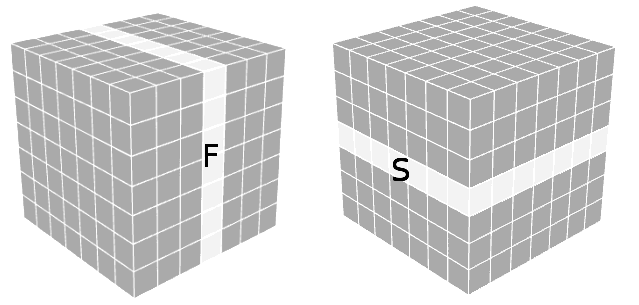}
    \caption{Filtering action by frames (orthoslice $F$) or slices (orthoslice $S$) on 
    the cubic dataset.}
    \label{fig:dataset}
\end{figure}

Asuming a cubic dataset with $N$ slices, $N$ rays and $N$ angles, it is easy to note that 
the computational complexity to apply a phase-retrieval strategy either by frames or slices on the entire cube, 
is $O(N^3 \log N)$. The disadvantage of using the frame-strategy is the need
to restore the cubic dataset prior to reconstruction, while the slice-strategy can be
easily included in  the stacking reconstruction process. Figure \ref{fig:dataset} ilustrates 
the difference between filtering by frames and slices. Although they never provide identical 
filtered datasets, the difference is bounded by a small constant, as we show in Proposition
\ref{prop:diff}.

The next Proposition provides a relation between our one-dimensional optimization strategy 
and the dispersion measure of the input data.
\begin{theo} \label{prop:disp}
The optimal curvature parameter $\ell$, which maximize the curvature of function $\xi_{\sf s}$ defined in (\ref{eq:xis}), is invertionally proportional to the cutoff frequency $\sigma_c$.
\end{theo}
\begin{proof}
We begin with definition (\ref{eq:xis}) and use the fact that $\partial_t^2 p_\ell = (g-p_\ell)/\ell$ for 
all $\theta \in [0,\pi]$ fixed, but arbitrary. Due to Parseval's identity we obtain for large values
of the cutoff frequency $\sigma_c$
\begin{eqnarray}
\xi_{\sf s}(\ell) &=& \frac{1}{\ell^2} \|p_\ell - g\|_2^2 
= \frac{1}{\ell^2} \| \hat{p_\ell} - \hat{g} \|_2^2 
\sim \frac{1}{\ell^2} \int_{-\sigma_c}^{\sigma_c} (\hat{p_\ell}(\sigma) - \hat{g}(\sigma) )^2 \mathrm d \sigma \\
&=& \frac{1}{\ell^2} \int_{-\sigma_c}^{\sigma_c} \left( \frac{1}{1+4\pi^2 \ell \sigma^2} \hat{g}(\sigma)- \hat{g}(\sigma) 
\right)^2 \mathrm d \sigma \\
&=&  \frac{1}{\ell^2} \int_{-\sigma_c}^{\sigma_c} \left( \frac{4\pi^2 \ell \sigma^2}{1+4\pi^2 \ell \sigma^2} \hat{g}(\sigma) \right)^2 \mathrm d \sigma \\
&=&  16 \pi^4 \hat{g}(\sigma^*)^2 \int_{-\sigma_c}^{\sigma_c}  \frac{\sigma^4}{(1+4\pi^2\ell \sigma^2)^2} \mathrm d \sigma \\
&=& 32 \pi^4 \hat{g}(\sigma^*)^2 \underbrace{\left[ \frac{3\sigma_c + 8\pi^2 \ell \sigma_c^2}{32 \pi^4 \ell^2 (4\pi^2\ell \sigma_c^2 + 1)} - \frac{ 3 \arctan \left( 2\pi \sqrt{\sigma_c} \right)}{2 (4\pi^2 \ell)^{5/2}}  \right]}_{D(\ell,\sigma_c)} \label{eq:disp}
\end{eqnarray}
Last equality is obtained using a straightforward integral calculation and a mean value of $\hat{g}(\sigma^*)^2$ for some $\sigma^*$ over the interval $(-\sigma_c, \sigma_c)$. It is easy to realize that function $D$ is asymptotically defined 
as $D(\ell,\sigma_c) \sim O(\sigma_c/\ell)$ so that
\[
D(\cdot,\sigma_c^{(1)}) \geq D(\cdot,\sigma_c^{(2)}) \geq  D(\cdot,\sigma_c^{(3)}) 
\]
for an increasing sequence of cutoff frequencies $\sigma_c^{(1)} \leq \sigma_c^{(2)} \leq \sigma_c^{(3)}$. From the above inequality and since $D$ has an assymptotic behaviour with $1/\ell$, the optimal curvature point $\ell_k^*$ from function $D(\cdot, \sigma_c^{(k)})$ determines an increasing sequence 
$\ell_1^* \leq \ell_2^* \leq \ell_3^*$, as shown in Figure \ref{fig:curvs}.
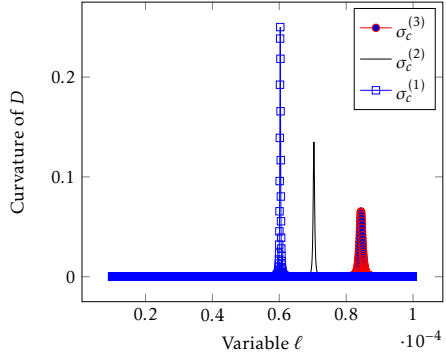
\begin{figure}[h]
    \centering
	\begin{tikzpicture}[scale=0.7]  
		\begin{axis}[
		  xlabel=Variable $\ell$,
		  ylabel=Curvature of $D$,
		  ]
		\addplot+[color=red,mark=*,legend entry=$\sigma_c^{(3)}$] table[x index=0,y index=1]{curvature.dat};
		\addplot+[color=black,mark=circle,legend entry=$\sigma_c^{(2)}$] table[x index=0,y index=2]{curvature.dat};
		\addplot+[color=blue,mark=square,legend entry=$\sigma_c^{(1)}$] table[x index=0,y index=3]{curvature.dat};
		\end{axis}
	\end{tikzpicture}		
	\caption{Curvature of function $D(\ell,\sigma_c^{(k)})$ - from Eq.(\ref{eq:disp}) - using $\sigma_c^{(1)} \leq \sigma_c^{(2)} \leq \sigma_c^{(3)}$. }
	\label{fig:curvs}
\end{figure}
\end{proof}

\medskip

A similar discussion applies for the two-dimensional case, changing kernel $\T_\ell$ by $\K_\ell$.
in the 2D frequency variable $\bm q$. Further strategies could be used, replacing operator $\mathcal Y$ in (\ref{eq:reg4}) by a generalized linear combination of first and second derivatives on the $t$-axis in order to obtain smoother sinograms, i.e.,
\begin{equation}
\mathcal Y = a_1 \frac{\partial}{\partial t} + a_2 \frac{\partial^2}{\partial t^2} \ \ \ 
\Rightarrow \ \ \ 
\mathcal Y^* = - a_1 \frac{\partial}{\partial t} + a_2 \frac{\partial^2}{\partial t^2}
\end{equation}
In this case, the resulting differential equation, with his associated Fourier representation
is
\begin{equation}
\begin{array}{lll}
\ell \left(- a_1^2 \frac{\partial^2}{\partial t^2}  + a_2^2 \frac{\partial^4}{\partial t^4} + 1 \right) p = g \\ \\
\ \ \Rightarrow \ \ 
\displaystyle \hat{p}(\sigma,\theta) = \frac{\hat{g}(\sigma,\theta)}{1 + 4\pi^2 \ell (a_1^2 + 4 \sigma^2 a_2^2) \sigma^2}
\end{array}
\end{equation}

\section{Regularized Filtered Backprojection $\times$ Phase-Retrieval}

The filtered backprojection inversion algorithm states that a given slice $s$ can be
reconstructed from a sinogram through the following inversion formula
\[
s(\bm x) = \bm B \bm F [g] (\bm x), \ \ \ \ \bm x \in D
\]
with $g = g(t,\theta)$ the measured image data. Operator $\bm B$ is the so-called backprojection operator, 
which is the adjoint of the Radon transform $\bm R$, i.e., the stacking operator over the family of 
x-rays passing through the sample. Operator $\bm F$ is a low-pass filtering operator, acting on the variable $t$, i.e., $\widehat{\bm F g}(\sigma) = |\sigma| \hat{g}(\sigma,\theta)$, for all $\theta$. The complex computational part of the \textsc{fbp} algorithm is the computation of $\bm B$ for all pixels in the region of interest $D$. It is an algorithm with complexity $O(N^3)$ with $N$ being the number of pixels for the reconstructed image. It was recently shown \cite{taormina} that $B$ can be computed with complexity $O(N^2\log N)$ using a polar representation. In this case, it is possible to proof \cite{bst} that a dual formulation of the Fourier-slice theorem is also valid; it was refered as a \emph{backprojection-slice-theorem} (\textsc{bst}). The theorem states that the Fourier of backprojection of any sinogram image can be computed in polar coordinates of the frequency domain using
\[
\widehat{\bm B[s]}(\sigma \cos\theta, \sigma \sin\theta) = \sigma^{-1} \hat{g}(\sigma, \theta), 
\]
with $(\sigma,\theta) \in \R_+ \times [-\pi,\pi]$. Further details about the above formula can be found in \cite{taormina}. 
The \textsc{bst} formula can be used to obtain an analytical solution of the standard Tikhonov regularization problem
\begin{equation}
 \begin{array}{lll} \label{eq:opt}
       \minimize_{s \in L_2} \ \ \|  \bm R \bm C s - g \|^2 + 4\pi^2 \ell \|s\| \\
       \end{array}
\end{equation}
with $\bm C$ the smoothingself-adjoint operator defined as
\begin{equation}
\bm C f (\bm x) = \int_{\R^2} \frac{f(\bm x')}{\|\bm x' - \bm x\|} \mathrm d \bm x' 
\end{equation}
In (\ref{eq:opt}), we are looking for the best smooth approximation to the measured
data. Here, $U_2$ is the $L_2$ space equiped with the $L_2$ norm. After exploring the Euler-Lagrange equations for the above optimization
problem, we obtain the following representation in the frequency domain for 
the solution $s$
\begin{equation} \label{eq:reg_fst}
\hat{s}(\sigma\cos\theta, \sigma\sin\theta) = \left( \frac{\sigma}{1+ 4\pi^2\ell^2 \sigma^2} \right) \hat{g}(\sigma, \theta)
\end{equation}
The proof of the above equation can be found with details in \cite{bst}, although using $\bm C$ as an identity operator. Since $\bm C$ is self-adjoint, the \textsc{bst} formula with further properties of the Fourier transform give us (\ref{eq:reg_fst}).  We note that (\ref{eq:reg_fst}) is a regularized version of the Fourier-Slice-Theorem and can be used to obtain $s$ explicitly through any gridding strategy \cite{marone}.  Applying (\ref{eq:reg_fst}) in the Fourier representation of $s(\bm x)$ we finally obtain a new representation for the reconstructed image $s$,
\begin{equation}
s(\bm x) = \int_\R \mathrm d\sigma \int_0^{\pi} \mathrm d \theta 
\left( \frac{|\sigma|}{1+4\pi^2 \ell \sigma^2} \right) \hat{g}(\sigma, \theta) |\sigma| e^{i \sigma \bm x \cdot \bm \xi_\theta} \label{eq:reg_fbp}
\end{equation}
Equation (\ref{eq:reg_fbp}) provides exactly the same reconstruction pattern
as a typical filtered backprojection reconstruction algorithm, but with the phase-retrieval kernel $\T_\ell$ acting on the sinogram $g$. In fact, we can generalize our regularized strategy
in the following representation
\begin{equation} \label{eq:reg_sol}
s_\ell(\bm x) = \bm B \bm F^2 [ \T_\ell g](\bm x), \ \ \ \bm x \in D
\end{equation}
Now, $\{s_\ell\}$ is a family of solutions of the optimization problem (\ref{eq:opt}), depending on the regularization parameter $\ell$. The filter function
$\T_\ell$ is defined in (\ref{eq:pr2}). Our regularized solution (\ref{eq:reg_sol}) depends explicitly on the computation of the backprojection operator $\bm B$ and either the \textsc{bst} formula, or different strategies could be used. 
Figure \ref{fig:1dFilters}.a presents the family filter $\{\T_\ell\}$ for three different values of $\ell$, while \ref{fig:1dFilters}.b the lowpass action of the filter $\bm F \T_\ell$ on the input sinogram $g$.
\begin{figure}
\centering
\begin{tabular}{cc}
(a)  $\bm \T_\ell(\sigma)$ & (b) $|\sigma| \T_\ell(\sigma)$ \\
\begin{tikzpicture}[scale=0.5]
\tikzstyle {every pin} = [ fill=yellow!50!white, rectangle ]
\tikzstyle {smallDot} = [ fill=red, circle, scale=0.5 ]
\begin{axis}
[grid=major,
axis on top=true,
axis x line=middle,
axis y line=middle,
legend pos=north east]
\addplot[black, ultra thick,domain=-2:2, legend entry=$0$] {(1.0)/(1 + 0 * abs(x))};
\addplot[black, ultra thick,domain=-2:2, legend entry=$0.02$, mark=*] {(1.0)/(1 + 39.4784176044 * 0.02 * abs(x))};
\addplot[black, ultra thick,domain=-2:2, legend entry=$0.05$, mark=square] {(1.0)/(1 + 39.4784176044 * 0.05 * abs(x))};
\end{axis}
\end{tikzpicture}
& 
\begin{tikzpicture}[scale=0.5]
\tikzstyle {every pin} = [ fill=yellow!50!white, rectangle ]
\tikzstyle {smallDot} = [ fill=red, circle, scale=0.5 ]
\begin{axis}
[grid=major,
axis on top=true,
axis x line=middle,
axis y line=middle,
legend pos=north east]
\addplot[red, ultra thick,domain=-2:2, legend entry=$0$] {(abs(x))/(1 + 0*x*x)};
\addplot[red, ultra thick,domain=-2:2, legend entry=$0.02$, mark=*] {(abs(x))/(1 + 39.4784176044 * 0.02 *x*x)};
\addplot[red, ultra thick,domain=-2:2, legend entry=$0.05$, mark=square] {(abs(x))/(1 + 39.4784176044 *0.05*x*x)};
\end{axis}
\end{tikzpicture}
\end{tabular}
\caption{One-dimensional kernels (a) $\T_\ell$ and (b) $|\sigma|\T_\ell$ for the determination of the filtered sinograms using a phase-retrieval filtered backprojection strategy.}
\label{fig:1dFilters}
\end{figure}
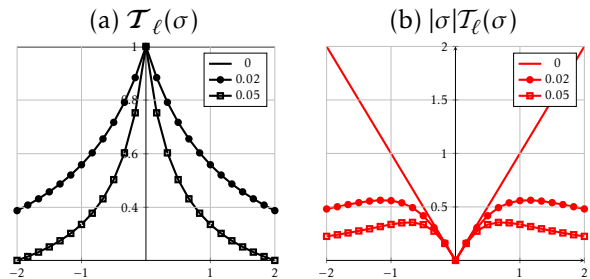

\section{Experimental validation} 

The curimat\~a fishegg sample, described in the Introduction, was exposed to the imaging beamline from the Brazilian synchrotron light source, using only 200 angles.
Figure \ref{fig:fish_original} presents a slice reconstructed using the transmission expectation maximization algorithm \cite{lange}. It is clear that there is a small absorption of the sample and low contrast. Hence, image segmentation 
is nearly impossible. Well known image segmentation techniques, commonly implemented in commercial software for 3D visualization and data analysis, grossly failed in this case of study. Region-based segmentation methods \cite{khan} use the pixel grayscale to separate the image into different regions and threshold method works well with a bimodal distribution of grayscales, which is clearly not seen in this case, see Figure \ref{fig:without}.(a). Watershed, a more robust region-based segmentation method, comes from the geological idea of valleys and ridges identification, which are related to the grey level of the image. In this specific case, since there is a high grayscale variation observed between neighbouring pixels - see Figure \ref{fig:without}.(b) - this methodology oversegment the image, creating many small different regions.

\begin{figure}[h]
\centering
\begin{tikzpicture}[zoomboxarray,
    zoomboxes below,
    zoomboxarray columns=3, 
    zoomboxarray rows=1,
    connect zoomboxes,
    zoombox paths/.append style={ultra thick, gray}]
    \node [image node] {\includegraphics[scale=0.5]{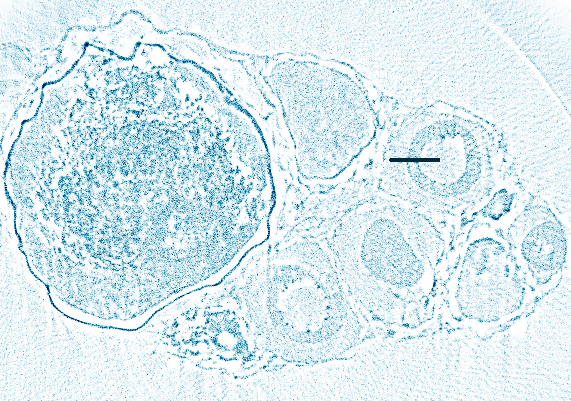}}; 
    \zoombox[magnification=3.2]{0.45,0.50};
    \zoombox[magnification=2]	{0.78,0.60}; 
    \zoombox[magnification=3.5]{0.95,0.37};
\end{tikzpicture}
\caption{Slice reconstruction of a fish egg sample. The zoomed regions the difference between phases of the sample. The reconstruction was obtained using the transmission expectation maximization algorithm on the measured data.}
\label{fig:fish_original}
\end{figure}

The reconstruction using the expectation maximization algorithm with the filtered sinogram  - using kernel $\T_\ell$ is presented in Fig.\ref{fig:fish_paganin}, where now contrast is considerably higher if compared with the unfiltered result of Fig.\ref{fig:fish_original}. Image segmentation of the new image is now much easier, and is presented in Fig.\ref{fig:segmentada}. Here, as the image contrast was enhanced, threshold tool was enough to separate most of the regions (e.g. contour region 1 and 2 of Figure \ref{fig:fish_paganin}). However, other regions still present high differences in grayscale (e.g. region 3 of Figure \ref{fig:fish_paganin}), and a combination of edge detection and fill interior tools had to be applied. The optimal value for $\ell$ was obtained using the curvature strategy described in Section \ref{sec:tiko}.

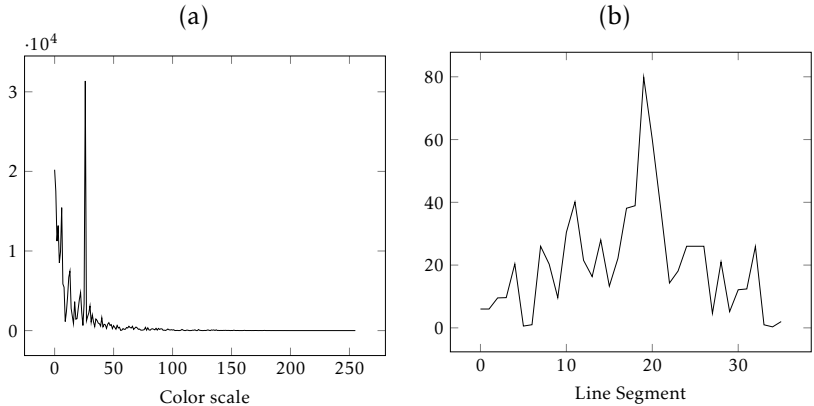
\begin{figure}[t]
    \centering
	\begin{tabular}{cc}
	    (a) & (b) \\
		\begin{tikzpicture}[scale=0.7]  
		\begin{axis}[
		  xlabel=Color scale,
		  ]
		\addplot[mark=none] table[x index=0,y index=1]{histo_without_reg.dat};
		\end{axis}
		\end{tikzpicture}
		&
 		\begin{tikzpicture}[scale=0.7]	
		\begin{axis}[
		  xlabel=Line Segment,
		  ]
		\addplot[mark=none] table[x index=0,y index=1]{corte_without_reg.dat};
		\end{axis}
		\end{tikzpicture}
	\end{tabular}
	\caption{(a) Histogram for the reconstructed image with normal transmission sinogram data, see Fig.\ref{fig:fish_original}.
	(b) Intensity plot along a line segment through the reconstructed image.}
	\label{fig:without}
\end{figure}

\begin{figure}[h]
\centering
\begin{tikzpicture}[zoomboxarray,
    zoomboxes below,
    zoomboxarray columns=3,
    zoomboxarray rows=1,
    connect zoomboxes,
    zoombox paths/.append style={ultra thick, gray}]
    \node [image node] {\includegraphics[scale=0.5]{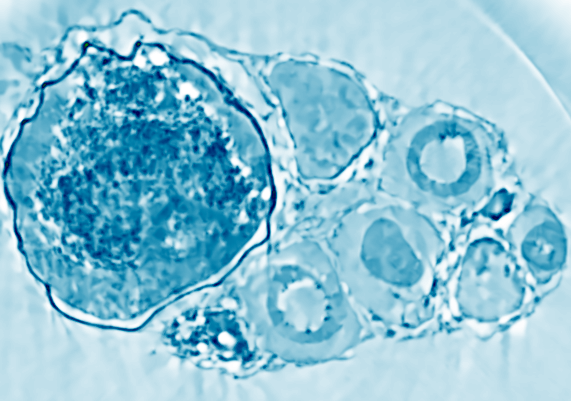}}; 
    \zoombox[magnification=3.2]{0.45,0.50}
    \zoombox[magnification=2]{0.78,0.60} 
    \zoombox[magnification=3.5]{0.95,0.37}
\end{tikzpicture}
\caption{Slice reconstructed of the fish egg sample using the transmission expectation maximization algorithm on the filtered data,  using the strategy described in (\ref{eq:pr2}) with kernel $\T_\ell$.}
\label{fig:fish_paganin}
\end{figure}

\begin{figure}[h]
\centering
\includegraphics[scale=0.3]{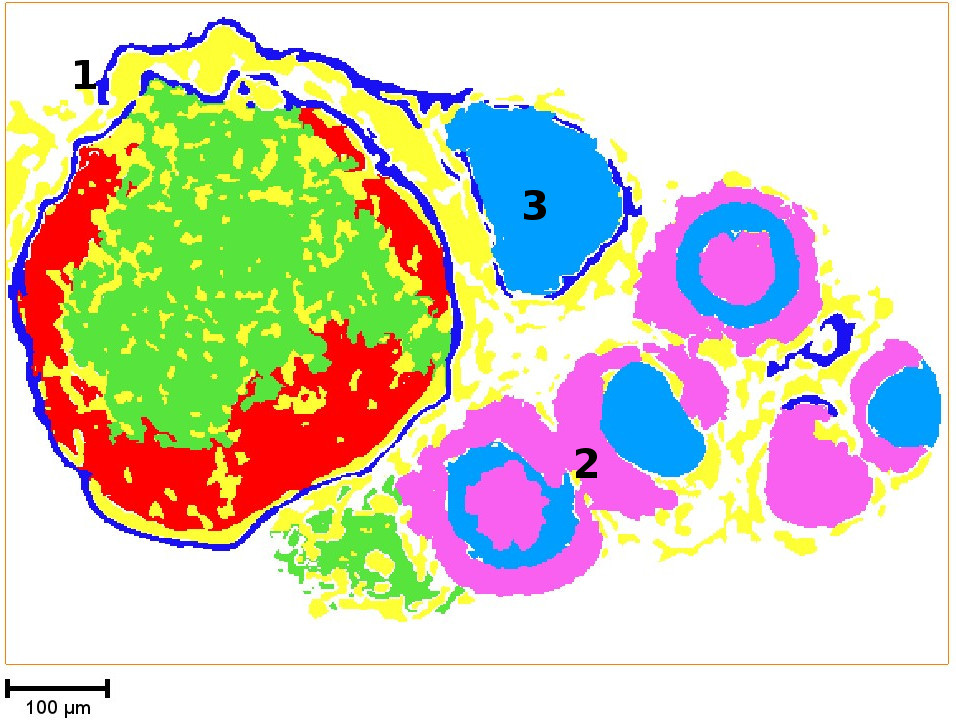} 
\caption{Segmentation of the reconstructed slice of Fig.\ref{fig:fish_paganin}
using the filtering strategy described in (\ref{eq:pr2}).}
\label{fig:segmentada}
\end{figure}

\begin{figure}[t]
    \centering
    \begin{tabular}{cc}
	    (a) & (b) \\
		\begin{tikzpicture}[scale=0.7]  
		\begin{axis}[
		  xlabel=Color scale,
		  ]
		\addplot[mark=none] table[x index=0,y index=1]{histo_with_reg.dat}; 
		\end{axis}
		\end{tikzpicture}
		&
 		\begin{tikzpicture}[scale=0.7]  
		\begin{axis}[
		  xlabel=Line segment,
		  ]
		\addplot[mark=none] table[x index=0,y index=1]{corte_with_reg.dat};
		\end{axis}
		\end{tikzpicture}
	\end{tabular}
	\caption{(a) Histogram for the reconstructed image, using the regularized sinogram, see Fig.\ref{fig:fish_paganin}; (b) Intensity plot along a line segment through the regularized image.}
	\label{fig:with}
\end{figure}
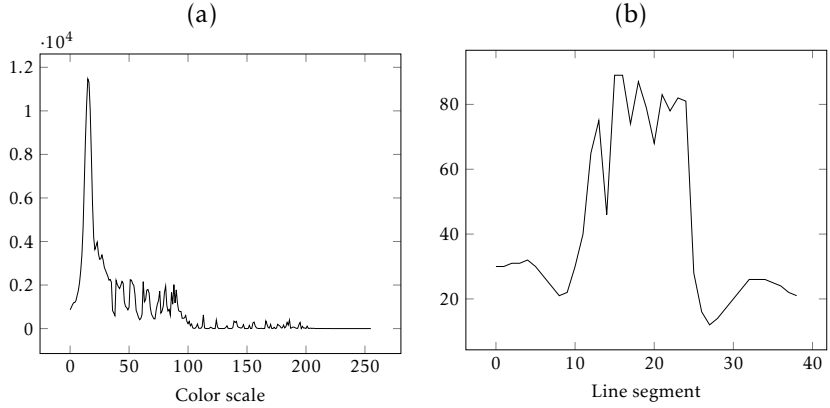

\begin{figure}
\centering
\begin{tabular}{cc}
(a) & (b) \\
\includegraphics[width=0.4\columnwidth]{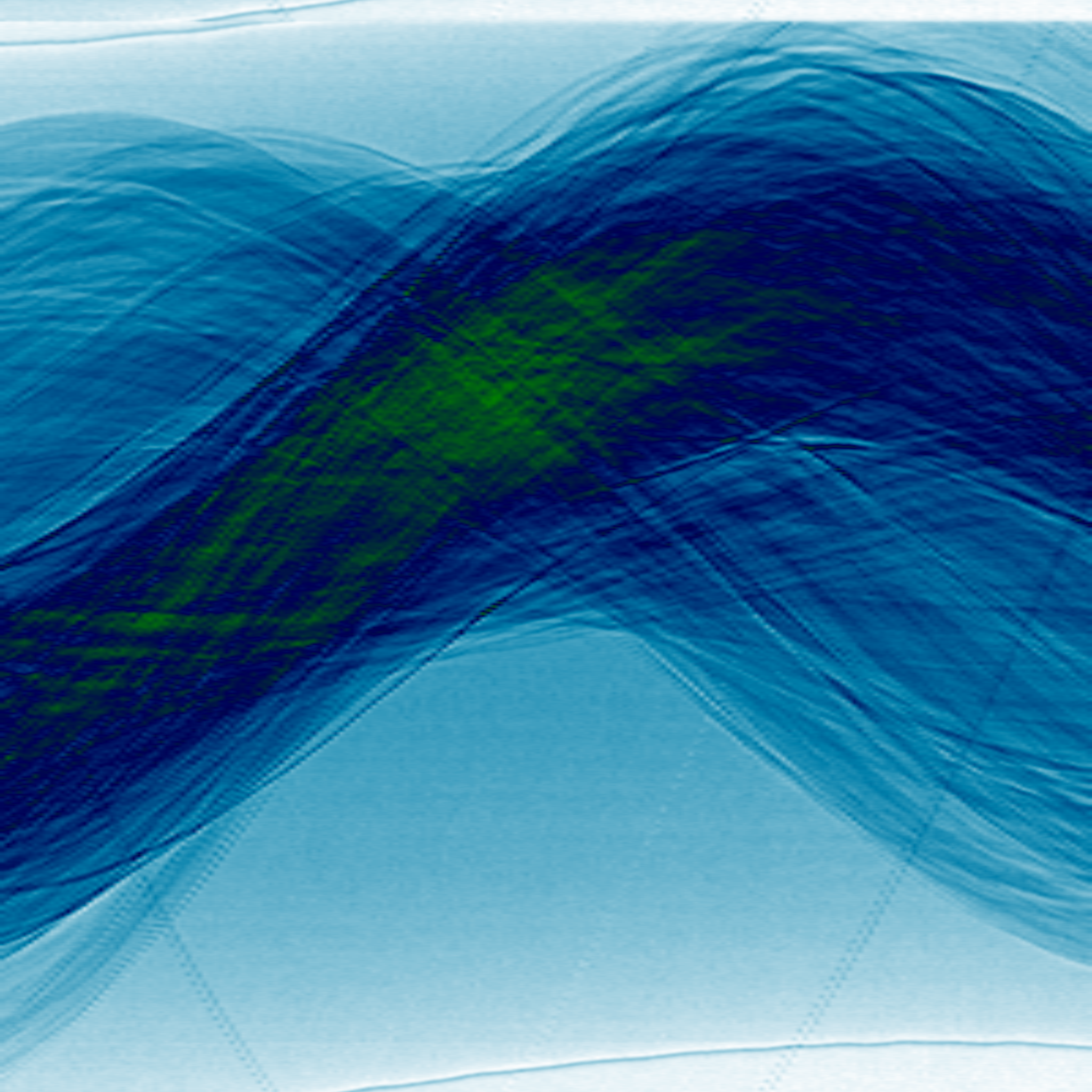}
&
\includegraphics[width=0.4\columnwidth]{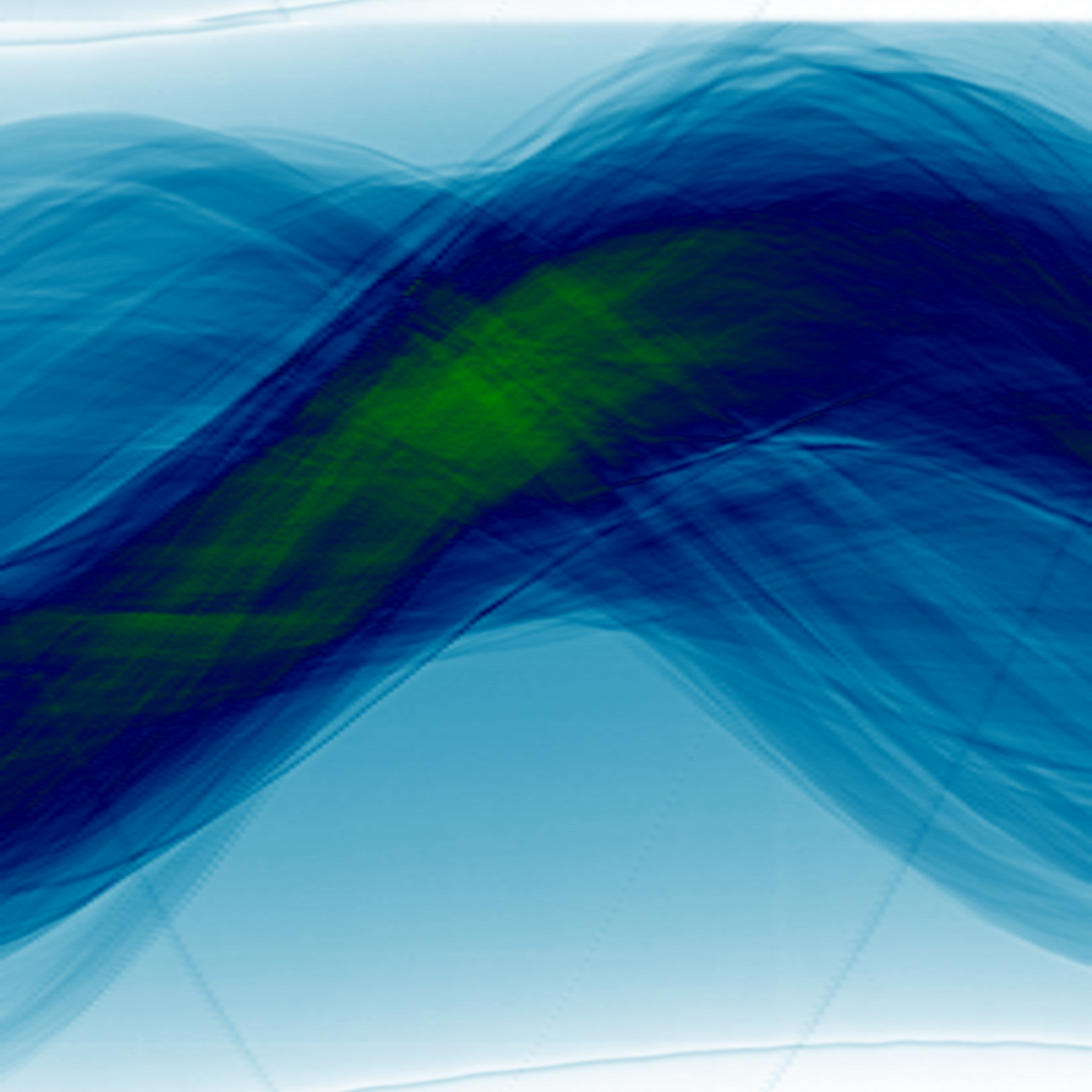}
\end{tabular}
\caption{Sinogram obtained with regularization by (a) frames and (b) slice.}
\label{fig:filter}
\end{figure}

Two filtered sinograms are depicted in Figure \ref{fig:filter} comparing the regularization strategy by frames and by slice. Although the sinograms are not equal, their difference is bounded by $0.15$, as shown in Table \ref{tab:values}. Also, Table \ref{tab:values} presents the optimal values using the one-dimensional curvature strategy.
\begin{table}[h]
\centering 
\begin{tabular}{l|c|c}
& Optimal $\ell$ ($\sf \mbox{meters}$) & $\displaystyle \max_{\bm y} |\Delta_{\ell} (\bm y)|$ \\
\hline
Frames & $0.000402175636558$ & $0.11386469548$ \\
\hline
Slice & $0.000463889940688$ & $0.11555728758$ \\
\end{tabular}
\caption{Comparison of obtained values for the parameter $\ell$ (in meters) at a fixed slice.}
\label{tab:values}
\end{table}
To validate the numerical curvature strategy, we expose the sample at 29 different distances $\{d_1, d_2,\ldots, d_{29}\}$ obtaining a normalized sequence of frames $\{f_k\}$,
\[
f_k(\bm y)=\frac{I_k(\bm y)}{I_0(\bm y)}, \ \ \ 4 \ \mbox{mm} \leq d_k \leq 282 \ \mbox{mm}
\]
A small distance $d_1$ indicates a near-contact propagation while 
for long distances like $d_{29}$, the phase becomes more visible. Each $f_k$ satisfies the approximation (\ref{eq:prop}) with a parameter $L_k$ that depends linearly on the distance
$d_k$ \cite{anka,burvall}, i.e.,
\begin{equation}
(-L_k \nabla^2 + I) e^{-p} = f_k, \ \ \ \ L_k = \frac{\delta}{\mu} d_k, \ \ \mu = \frac{4\pi\beta}{\lambda}
\end{equation}
where $\lambda$ is the wavelength. First and last frame are presented in Figure \ref{fig:frames}, where edge enhancement is  clearly visible for high distances in the Fresnel regime.
\begin{figure}[h]
\centering
\begin{tabular}{cc}
    (a) & (b) \\
    \includegraphics[scale=0.4]{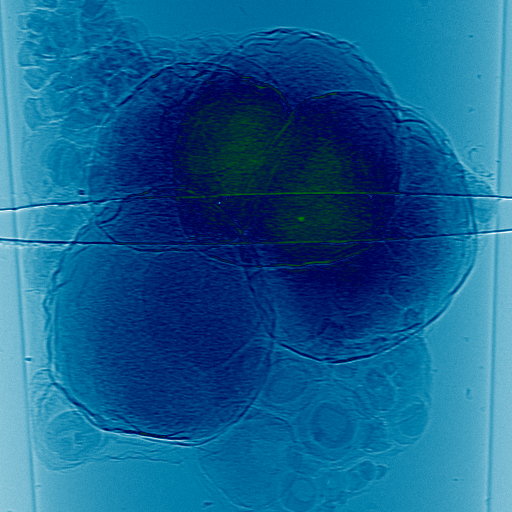} 
    &
    \includegraphics[scale=0.4]{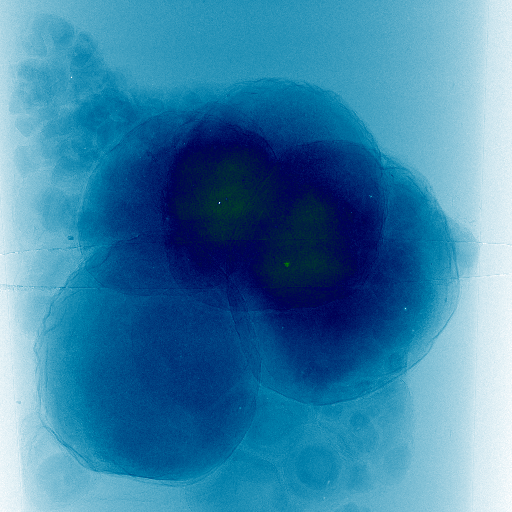} 
\end{tabular}
\caption{Projection at different distances. Higher distance of 282 mm at left, and 4 mm at right.}
\label{fig:frames}
\end{figure}

According to our numerical scheme, we can restore an approximation $\ell_k$ of $L_k$
computing the point of maximum curvature of the function $\xi_{\sf f}$ in (\ref{eq:curvs}), i.e,
\[
\ell_{k} = \argmax_{\ell \geq 0} \kappa_{\sf f}(\ell; f_k), \ \ \ 
\]
The sequence $\{\ell_k\}$ as a function of $\{d_k\}$, obtained for this experiment,
is presented in Figure \ref{fig:restored}.a. Since $\ell_k$ is not linear with respect to $d_k$, 
as predicted by theory, we investigate the effect of changing the 
cutoff-frequency $\sigma_c$ defined as
\begin{equation} \label{eq:m}
q_c = \frac{m}{2\Delta}
\end{equation}
where $\Delta$ is the sampling rate on the square domain for $\bm y$. Here, $m$ is a positive
constant and the mesh for the two-dimensional frequency variable $\bm q$ is such that
$\| \bm q \|_\infty \leq q_c$. Number $m$ is proportionally inverse to the
number $\ell$, as given in Proposition \ref{prop:disp}. In this case,
The blinded restoration (i.e., without knowing the true value of $L$) using the optimization strategy presented in the plot of Figure \ref{fig:restored}.a was obtained using $m=1$.

\begin{figure}[h]
    \centering
	\begin{tabular}{cc}
        (a) & (b) \\
    	\begin{tikzpicture}[scale=0.7]  
		\begin{axis}[
		  xlabel=Distance,
		  ]
		\addplot+[only marks,mark size=2.9pt] table[x index=0,y index=1]{matrix_L.dat};
		\end{axis}
		\end{tikzpicture}
		&
 		\begin{tikzpicture}[scale=0.7]	
		\begin{axis}[
		  xlabel=Distance]
		\addplot[only marks,mark size=2.9pt] table[x index=0,y index=1]{matrix_L_analytical.dat};
		\end{axis}
		\end{tikzpicture}
	\end{tabular}
	\caption{Wrong blinded restoration values of $\ell$, for an arbitrary value of $m=1$ (Eq.(\ref{eq:m})). As a function of distance, $\ell$ is not linear, contradicting the  theory. (a) Using real data with 29 distances and (b) a similar simulated one-dimensional example.}
	\label{fig:restored}
\end{figure}
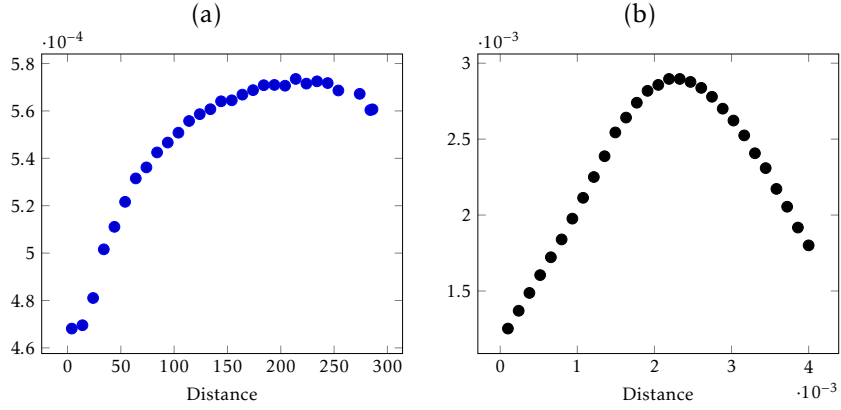

In order to investigate the effect of variable $m$, which defines the cutoff-frequency
(\ref{eq:m}), we investigate our approach using a one-dimensional example. Let 
$p(t) = \mbox{rect}(t)$ be the rectangular function representing the electronic density, i.e.,
$p(t) = \frac{1}{2w}[ \mbox{sign}\left(t+\frac{1}{2}\right) - \mbox{sign}\left(t-\frac{1}{2}\right)]$. We can
approximate $p$ using the following approximation for the $\mbox{sign}(\cdot)$ function, $s_n(t) = t/\sqrt{t^2 + \frac{1}{n}}$, i.e.,
\begin{equation}
p(t) \approx 
p_n(t) = \frac{1}{2w}\left[ s_n\left(x+\frac{1}{2}\right) - s_n\left(x-\frac{1}{2}\right) \right]
\end{equation}
We have used $w=300$ and $n=10^4$ as depicted in Figure \ref{fig:Idet1d}.a; 
in this case, the propagation $I_L(t)$ is determined by
\begin{eqnarray}
I_L(t) &=& \left( - L \frac{\partial^2}{\partial t^2} + 1 \right) e^{-p_n(t)} \\
&=& \left( - \ell [ p''_n(t) + p'_n(t)^2 ] + 1 \right) e^{-p_n(t)} 
\end{eqnarray}
with derivatives $\{ p_n'', p_n'\}$ known \emph{a priori}. Function $I_L$ is shown in Figure \ref{fig:Idet1d}.b
using value $L = d\delta/\mu = 0.0163522409163$ with $d=500 \times 10^3$ mm, $\delta=1.043 \times 10^{-6}$, $\beta=3.553 \times 10^{-10}$, $\lambda=1.4 \times 10^{-10}$ and $\mu= 4\pi \beta/\lambda$.
\begin{figure}[h]
    \centering
    \includegraphics[width=0.48\columnwidth]{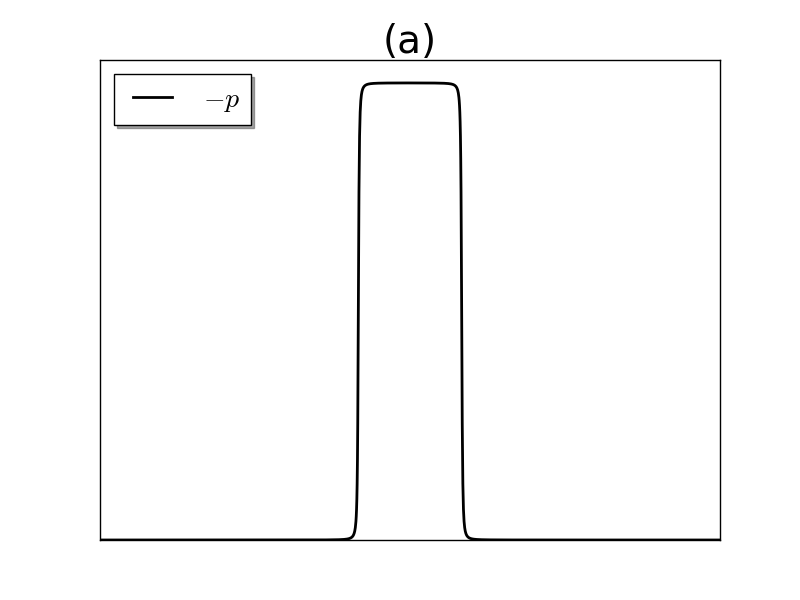}
    \includegraphics[width=0.48\columnwidth]{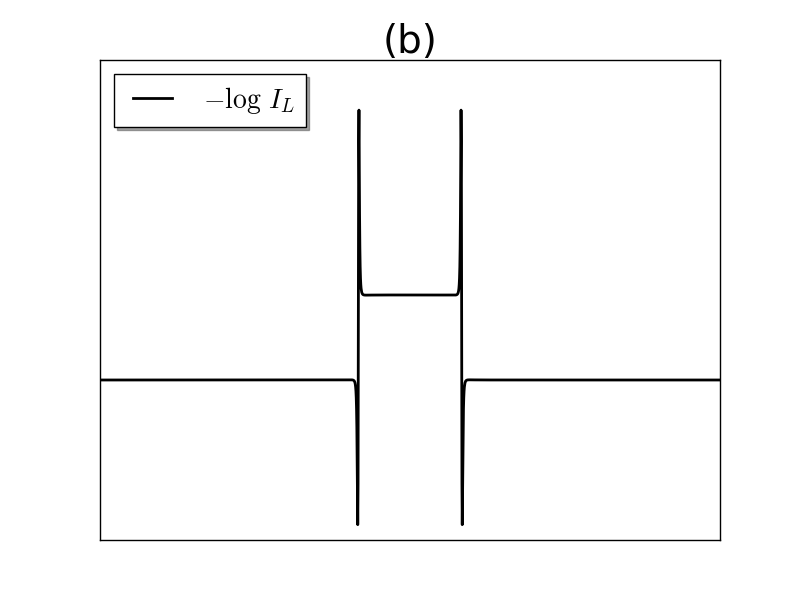}
    \caption{One-dimensional simulated example. (a) Electronic density function; (b) Measured function $-\log I_L$ with
    $L$ being fixed.}
    \label{fig:Idet1d}
\end{figure}
Let us denote the recovered signal as $p^{(\ell,m)}(t)$. Some examples with different
values of the pair $(\ell,m)$ are presented in Figure \ref{fig:lm}. Columns (a), (b) and (c) are for constant values of $\ell = 0.00163522409163$, $\ell=0.0825788166274$ and 
$\ell=0.163522409163$, respectively. Rows are varying from top to bottom with  $m = 0.15$, $m=0.25$ and $m=0.35$. 
\begin{figure}[h]
\begin{tabular}{c !{\vrule width -18pt} c !{\vrule width -18pt} c !{\vrule width -18pt}   }
(a) & (b) & (c) \\
\includegraphics[width=0.33\columnwidth]{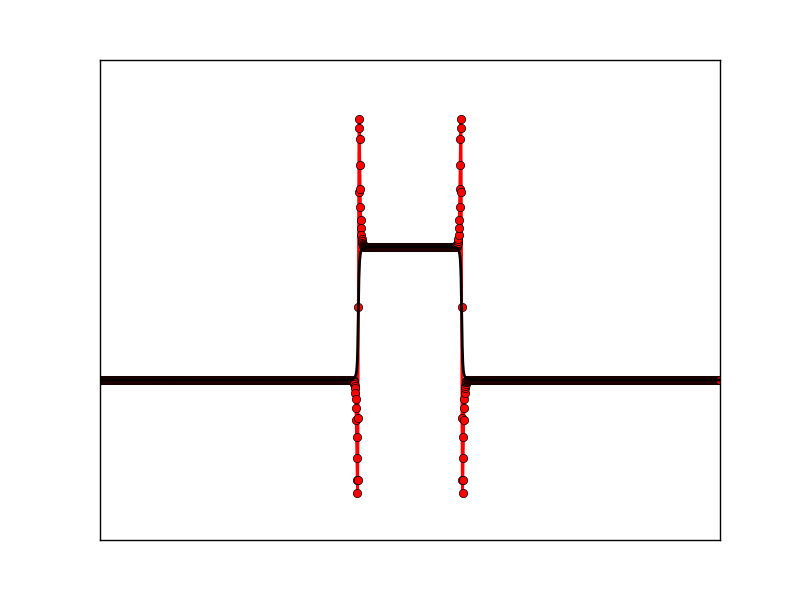}
&
\includegraphics[width=0.33\columnwidth]{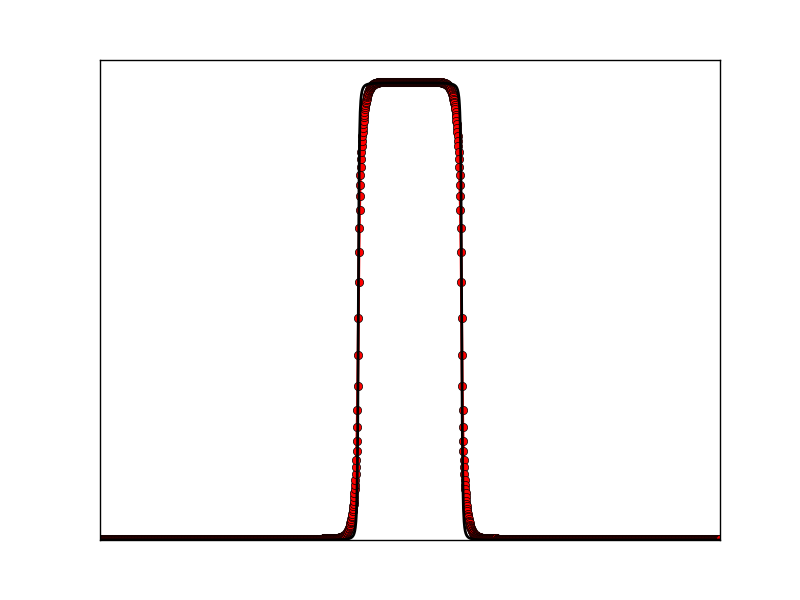}
& 
\includegraphics[width=0.33\columnwidth]{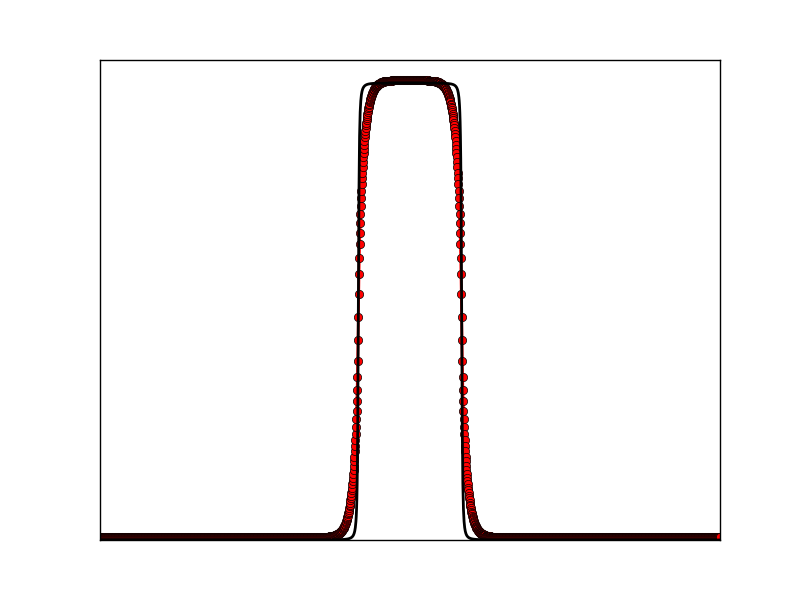}
\\
\includegraphics[width=0.33\columnwidth]{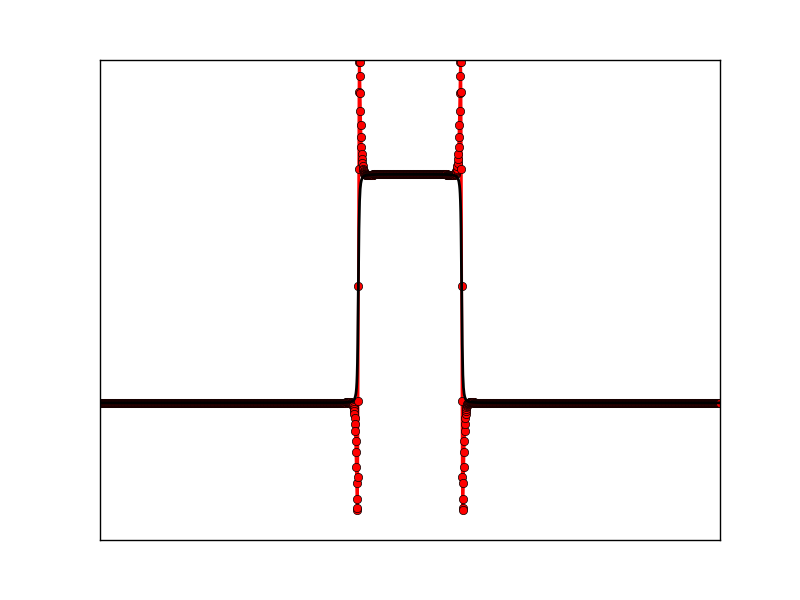}
&
\includegraphics[width=0.33\columnwidth]{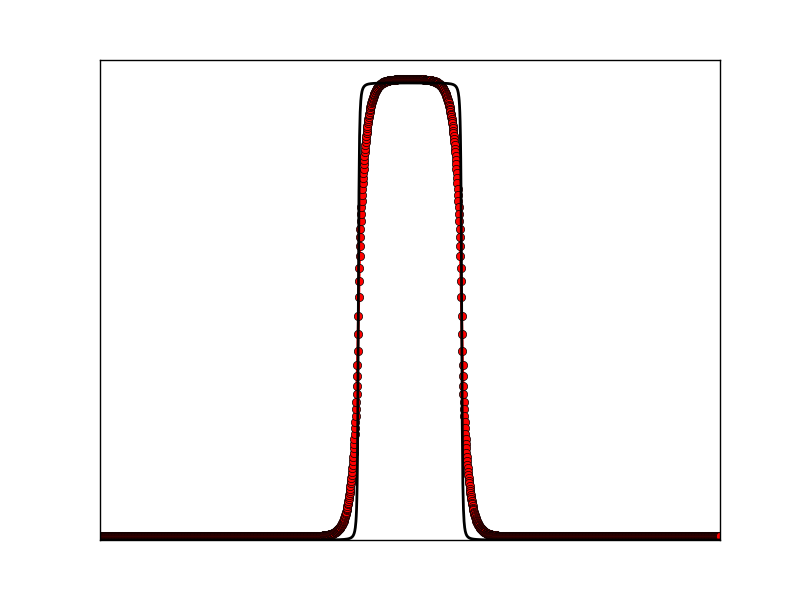}
& 
\includegraphics[width=0.33\columnwidth]{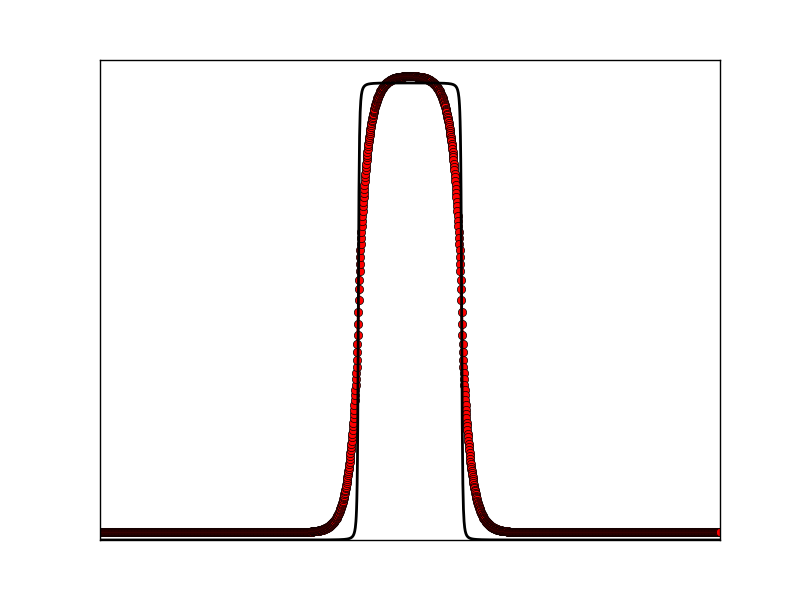}
\\
\includegraphics[width=0.33\columnwidth]{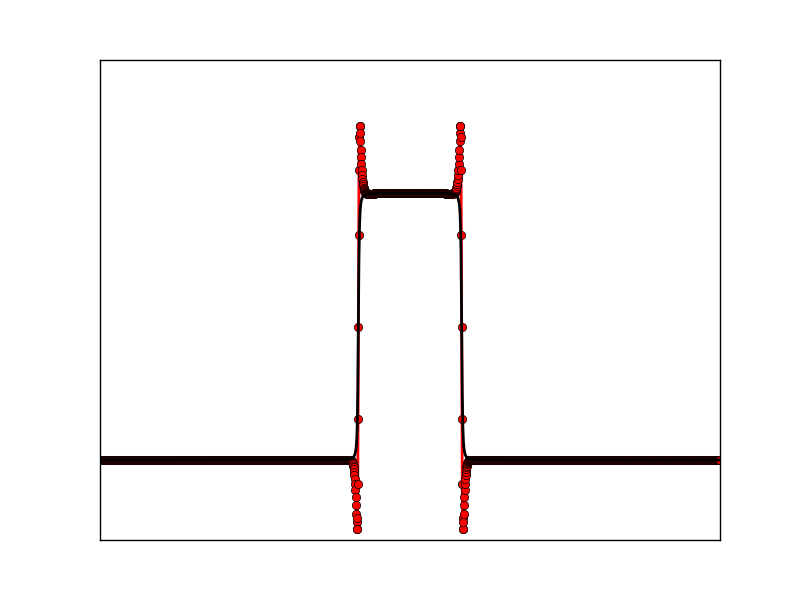}
&
\includegraphics[width=0.33\columnwidth]{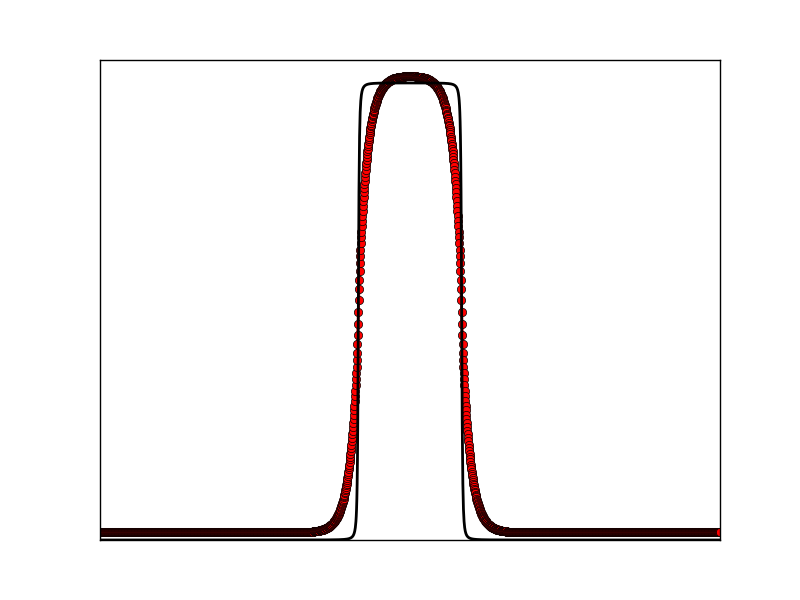}
& 
\includegraphics[width=0.33\columnwidth]{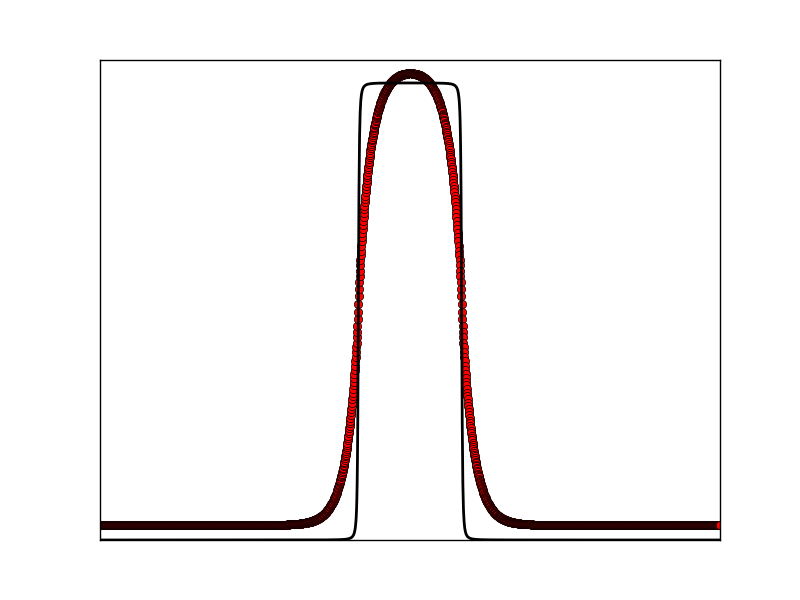}
\end{tabular}
\caption{Examples of restored signal $p^{(\ell,m)}$ with different values of 
$\ell$ and $m$. See text for details.}
\label{fig:lm}
\end{figure}

Taking the result of Figure \ref{fig:lm}, we observe (due to the intermediate
value Theorem) that exist a value of $\ell$ between 0.00163522409163
and 0.163522409163 that restore the signal $p$. However, the restored value for
an arbitrary value of $m$ possibly does not have any relation with the one 
obtained by the optimization strategy. Let $\{m_j\}$ be the sequence of values 
for $m$, which controls the cutoff-frequency (\ref{eq:m}). 
The sequence of restored values $\{\ell(m_j)\} \equiv \{\ell_j\}$ is presented in Figure 
\ref{fig:ivt}.a, where it is clear that exist a value of $m$ such that $\ell_j = L$. 

Now, computing the right $m$ for each distance value $d_k$ using a force brute 
algorithm give us the sequence $m_k$ depicted in Figure \ref{fig:lm}.b. For an 
increasing value of distance, $\ell_k$ behaves linearly as predicted by theory 
and presented in Figure \ref{fig:lm}.c and \ref{fig:lm}.d.
It is important to note that $m_k$ versus $d_k$ behave as predicted in 
Proposition \ref{prop:disp}, i.e., $\ell_k \sim 1 /\sqrt{m_k}$. 
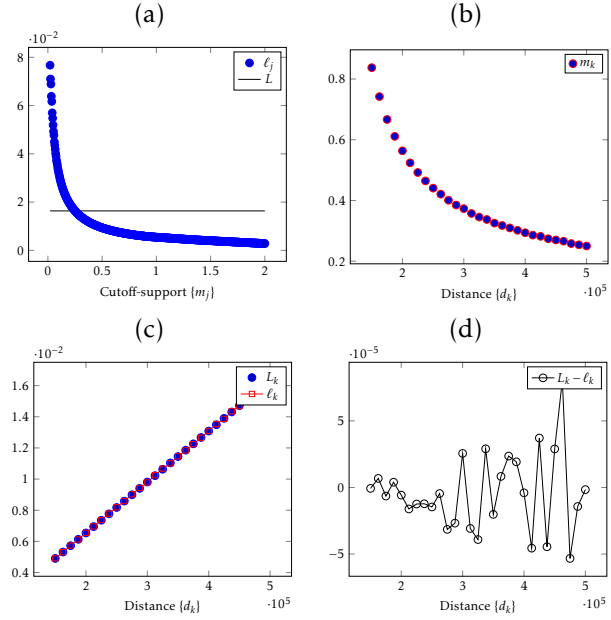
\begin{figure}[h]
\centering
\begin{tabular}{cc}
    (a) & (b) \\
	\begin{tikzpicture}[scale=0.5]
		\begin{axis}[
		  xlabel={Cutoff-support $\{m_j\}$},
		  ]
		\addplot +[only marks, mark size=2.9pt, legend entry={$\ell_j$},color=blue] table[x index=0, y index=1]{ivt.dat};
		\addplot[mark=none, legend entry={$L$}] table[x index=0, y index=2]{ivt.dat};
		\end{axis}
	\end{tikzpicture}
	& 
	\begin{tikzpicture}[scale=0.5]
		\begin{axis}[
		  xlabel={Distance $\{d_k\}$},
		  ]
		\addplot+[only marks, mark size=2.9pt, legend entry={$m_k$},color=red] table[x index=0, y index=3]{valuesofm.dat};
		\end{axis}
	\end{tikzpicture}
	\\
	(c) & (d) 
	\\
	\begin{tikzpicture}[scale=0.5]
		\begin{axis}[
		  xlabel={Distance $\{d_k\}$},
		  ]
		\addplot+[only marks, mark size=2.9pt, legend entry={$L_k$}] table[x index=0, y index=1]{valuesofm.dat};
		\addplot+[mark=square, legend entry={$\ell_k$}] table[x index=0, y index=2]{valuesofm.dat};
		\end{axis}
	\end{tikzpicture}
	& 
	\begin{tikzpicture}[scale=0.5]
		\begin{axis}[
		  xlabel={Distance $\{d_k\}$},
		  ]
         \addplot+[mark size=2.9pt, mark=o, legend entry={$L_k-\ell_k$},color=black] table[x index=0, y expr=\thisrowno{1}-\thisrowno{2}]{valuesofm.dat};
		\end{axis}
	\end{tikzpicture}
\end{tabular}
\caption{(a) Relation between the physical parameter $L$ and the sequence $\{\ell_j\}$ for different cutoff frequencies $\{m_j\}$. (b) Optimal cutoff parameter versus distance between sample and detector. (c) Sequence of restored parameter $\{\ell_k\}$ matching with the original sequence $\{L_k\}$. (d) Error between restored sequence $\{\ell_k\}$ and $\{L_k\}$  }
\label{fig:ivt}
\end{figure}

In the last experiment we expose a sample made of a thick kapton tape to the imaging beamline, using $201$ angles and a distance of $d=220 \times 10^3$ mm, each with an exposure time of 3 seconds. The idea was to test the one-dimensional optimization algorithm by frames. A given slice reconstructed with the expectation algorithm maximization is presented in Figure \ref{fig:ksemr}. The constrast is considerably 
poor inside the sample due to low absorption of the beam on the sample. Applying the 
one-dimensional strategy by frames with $m = 1$, we find an optimal $\ell^* = 0.0013871$, resulting in the reconstruction shown in Figure \ref{fig:kcomr}. According to the database from \textsc{nist} \cite{nist}, since kapton is a polyimide film with 
that $\delta = 3.9 \times 10^{-6}$ and $\beta = 7.02 \times 10^{-9}$ we obtain 
the theoretical $L$ given by $L = 0.0013617$, which is satisfactory to the $\ell^*$
obtained by our algorithm.
\begin{figure}[h]
\centering
\begin{tikzpicture}[zoomboxarray,
    zoomboxes below,
    zoomboxarray columns=3,
    zoomboxarray rows=1,
    connect zoomboxes,
    zoombox paths/.append style={ultra thick, gray}]
    \node [image node] {\includegraphics[scale=0.25]{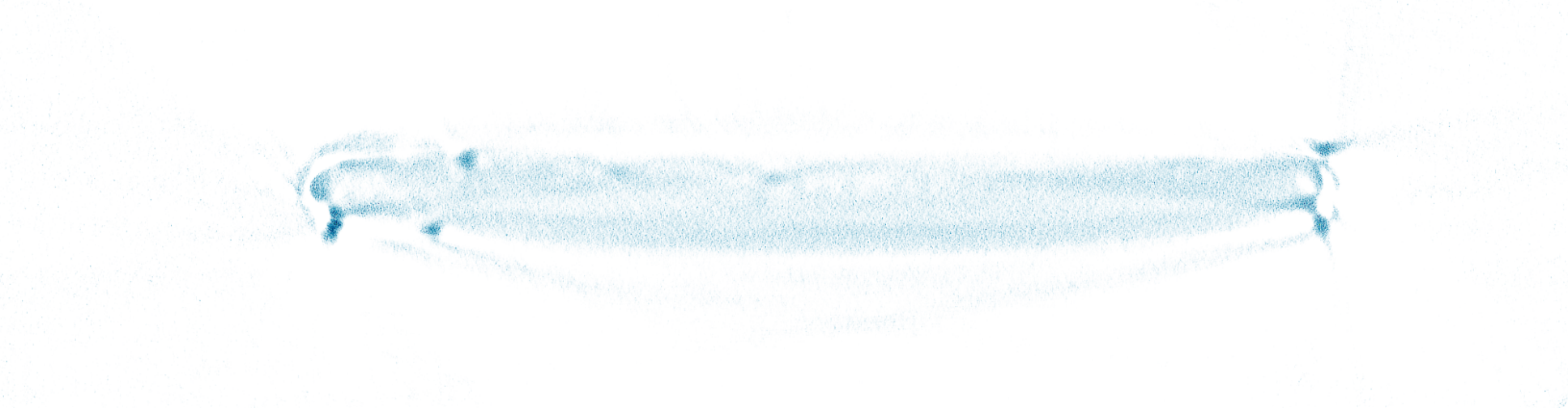}}; 
    \zoombox[magnification=2.3]{0.28,0.5}
    \zoombox[magnification=2.5]{0.53,0.50}
    \zoombox[magnification=2]{0.76,0.50}
\end{tikzpicture}
\caption{Reconstructed slice of a thick kapton sample using the expectation maximization
algorithm without regularization.}
\label{fig:ksemr}
\end{figure}

\begin{figure}[h]
\centering
\begin{tikzpicture}[zoomboxarray,
    zoomboxes below,
    zoomboxarray columns=3,
    zoomboxarray rows=1,
    connect zoomboxes,
    zoombox paths/.append style={ultra thick, gray}]
    \node [image node] {\includegraphics[scale=0.25]{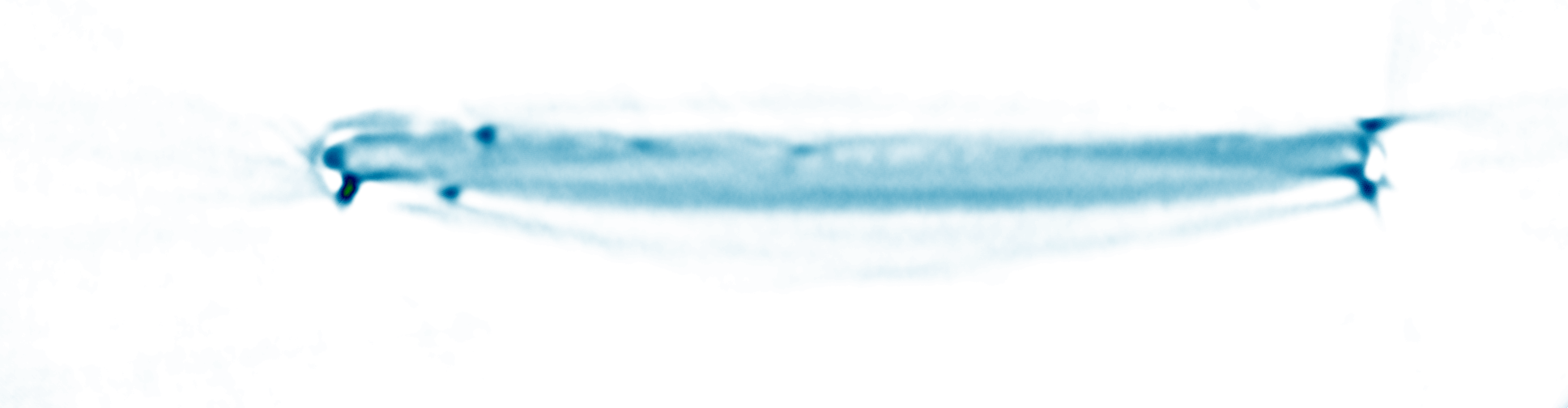}}; 
    \zoombox[magnification=2.3]{0.29,0.6}
    \zoombox[magnification=2.5]{0.56,0.58}
    \zoombox[magnification=2]{0.79,0.55}
\end{tikzpicture}
\caption{Reconstructed slice of a thick kapton sample using the expectation maximization
algorithm with regularization on the sinogram with kernel $\T_\ell$.}
\label{fig:kcomr}
\end{figure}

\section{Conclusion}

In this manuscript we have described a strong relation between in-line phase retrieval methods and a Tikhonov regularization for the imaging of light samples in the Fresnel regime. Considered as a regularized variational problem in an appropriate functional space, we provide a simple one-dimensional algorithm, which aims to restore the physical parameter $\ell$, usually given empirically, and needed to restore the phase in the near contact regime. An automatic algorithm capable to restore the right convolution parameter improve quality of image segmentation, which depends on the contrast 
of the reconstructed image. However, the one-dimensional algorithm relates directly to the cutoff-frequency $m$, also inversely proportional to the dispersion of the measured data. Even though $\ell$ and $m$ are also related, we claim that exist an optimal
$m^*$ for which the \textsc{1d} algorithm provide a best approximation of the appropriate physical parameter $\ell^*$. An algorithm to find $(\ell^*, m^*)$ was
out of the scope of this manuscript.

\appendix

\section{Appendix}

\subsection{Fr\'echet Derivatives}

Let $\|\cdot\| = \sqrt{\langle \cdot ,  \cdot \rangle}$ be the usual
$L^2$ norm. We compute the Fr\'echet derivative of the following 
nonlinear functionals $E$ and $R$, acting on a function $p = p (\bm y)\in L^2$
with $\bm y \in D$:
\begin{equation}
E(p) = \|e^{-p} - e^{-g} \|^2
\end{equation}
and
\begin{equation}
R(p) = \int \left\| \nabla \left( e^{-p(\bm y)}\right)\right\|^2 \mathrm d \bm y
\end{equation}

\bigskip
\noindent {$\bullet$ \bf Function $E$:}
    
\bigskip 

\noindent {\it First Derivative:}
\medskip    
    
    For all $h \in L^2$, the 
    error $\Delta E(p) = E(p+ h) - E(p)$ is given by
    \begin{equation*}
    \begin{array}{lll}
    \Delta E(p) = \|e^{-p}e^{-h}\|^2 - \|e^{-p}\|^2 - 2\langle e^{-p}e^{-h},e^{-g}\rangle + 2 \langle e^{-p}, e^{-g} \rangle \\
    \displaystyle \ \ \ = \int \mathrm e^{-2p(\bm y)} [e^{-2h(\bm y)} - 1] \mathrm d\bm y - 2
    \int e^{-p(\bm y)} e^{-g(\bm y)} [e^{-h(\bm y)} - 1] \mathrm d \bm y
    \end{array}
    \end{equation*}
    Since $e^{-h(\bm y)} = 1 - h(\bm y) + o(\|h\|^2)$ for $\bm y \in D$ 
    we finally get
    \begin{equation} \label{eq:ffinal}
    \Delta E(p) = 2 \int e^{-p(\bm y)} [e^{-g(\bm y)} - e^{-p(\bm y)} ] h(\bm y) \mathrm d\bm y + o(\|h\|^2)
    \end{equation}
    The definition of the Fr\'echet derivative $E'(p)$ comes from (\ref{eq:ffinal}) since
    \[
    \lim_{h \to 0} \frac{\Delta E(p) - E'(p) h}{\|h\|} = \lim_{h\to 0} o(\|h\|^2) = 0
    \]

\bigskip 

\noindent {\it Second Derivative:}
\medskip

    Second Fr\'echet derivative is a bilinear operator $E''(p)$ acting on a pair $(v,w) \in L^2 \times L^2$ as
    \begin{equation} \label{eq:2ndFrechet}
    E''(p)(v,w) = \lim_{t\to 0} \frac{R'(p+tw)v - R'(p) v}{t}
    \end{equation}
    Using the same reasoning, we can easily verify that 
    \[
    E''(p)(v,w) = - 2 \int e^{-p(\bm y)} v(\bm y) [ e^{-p(\bm y)} - 2 e^{-p(\bm y)}] w(\bm y) \mathrm d \bm y
    \]
    The approximation $e^{-tw} = 1 - tw + o(t^2)$ was used to obtain the above formula.

\bigskip
\noindent {$\bullet$ \bf Function $R$:}

\bigskip 

\noindent {\it First Derivative:}
\medskip

    Let us consider $F$ the following operator
    \[
    F(u) = \int_{D} \|\nabla u(\bm y)\|^2 \mathrm d \bm y,
    \]  
    which has Fr\'echet derivative given by $F'(u) = - 2 \nabla^2 u$, acting 
    as a linear operator in the following sense
    \[
    F'(u) h = -2 \int \nabla^2 u (\bm y) h(\bm y) \mathrm d \bm y    
    \]
    Now, our functional $R$ is related to $F$ as 
    \[
    R(p) = \int \|\nabla(e^{-p(\bm y)})\|^2 \mathrm d \bm y = F(e^{-p})
    \]
    or
    \[
    R( L (u ) ) = F(u),  \ \ \ L(u) = -\ln u
    \]
    Hence, using the chain rule in the Fr\'echet sense, 
    \[
    R'(L(u) ) \circ L'(u) = F'(u) \iff R'(-\ln u) \circ L'(u) = -2 \nabla^2 u
    \]
    Since $L$ is a also a composition, its easy to realize that $L'(u) = -1 / u$,
    therefore 
    \[
    R'( - \ln u ) \circ \left( \frac{1}{u} I \right) = - 2 \nabla^2 u
    \]
    with $I$ the identity operator. Replacing $p = -\ln u$ we finally obtain
    \[
    R'(p) \circ (e^p I) = -2\nabla^2(e^{-p}) \ \Rightarrow \ R'(p) = -2 e^{-p} \nabla^2(e^{-p})
    \]
    As a linear operator, the action of $R'(p)$ is given by
    \begin{equation} \label{eq:1stFofR}
    R'(p)h = -2 \int e^{-p(\bm y)} \nabla^2 (e^{-p(\bm y)}) h(\bm y)\mathrm d \bm y
    \end{equation}

\bigskip 

\noindent {\it Second Derivative:}
\medskip

    Using (\ref{eq:1stFofR}) and (\ref{eq:2ndFrechet}) we obtain 
    \[
    \begin{array}{lll}
    R''(p)(v,w) = 2 \displaystyle \int v(\bm y) e^{-p(\bm y)} \nabla^2 w(\bm y) + \\
    \ \ \ \displaystyle w(\bm y)e^{-p(\bm y)} \nabla^2(e^{-p(\bm y)}) v(\bm y) \ \ \mathrm d\bm y
    \end{array}
    \]
    This derivative was obtained using the approximation $e^{-tw} = 1 - tw + o(t^2)$ in the 
    definition (\ref{eq:2ndFrechet}).

\bigskip
\bigskip

\section*{Acknowledgments}

Staining  of  the  Curimat\~a  fish  egg sample was kindly provided by Juliana Martins S. Silva. We  also  thank Nikolay Koshev  and  Jo\~ao C.  Cerqueira  for  helping  with  the  implementation  of  some  of  the  methods  described  in  this manuscript.

\bigskip
\bigskip


\bibliography{main.bib}{}
\bibliographystyle{unsrt}

\end{document}